\documentclass{amsart}
\usepackage[margin=1.3in]{geometry}
\usepackage{amsmath}
\usepackage{amssymb}
\usepackage{mathtools}
\usepackage{graphicx}
\usepackage{tikz}
\usepackage{color}
\usepackage[all]{xy}
\usepackage{amsthm}
\usepackage{verbatim}
\usepackage[normalem]{ulem}

\usetikzlibrary{arrows,shapes,positioning}
\usetikzlibrary{decorations.markings}
\tikzstyle arrowstyle=[scale=1]
\tikzstyle directed=[postaction={decorate,decoration={markings,
    mark=at position .55 with {\arrow[arrowstyle]{stealth}}}}]
\tikzstyle ddirected=[postaction={decorate,decoration={markings,
    mark=at position .45 with {\arrow[arrowstyle]{stealth}},
    mark=at position .55 with {\arrow[arrowstyle]{stealth}}}}]
\tikzstyle reverse directed=[postaction={decorate,decoration={markings,
    mark=at position .55 with {\arrowreversed[arrowstyle]{stealth};}}}]
\tikzstyle reverse ddirected=[postaction={decorate,decoration={markings,
    mark=at position .55 with {\arrowreversed[arrowstyle]{stealth};},
    mark=at position .65 with {\arrowreversed[arrowstyle]{stealth};}}}]

\usepackage[colorlinks,linkcolor=blue,citecolor=blue,pdfstartview=FitH]{hyperref}

\theoremstyle{plain}

\newtheorem{para}{}[section]
\newtheorem{thm}[para]{Theorem}
\newtheorem{prop}[para]{Proposition}
\newtheorem{lemma}[para]{Lemma}

\newtheorem{fact}[para]{Fact}

\theoremstyle{remark}

\newtheorem*{remark}{Remark}

\theoremstyle{definition}

\newcommand{\C}{\mathbb{C}}

\newcommand{\Q}{\mathbb{Q}}
\newcommand{\R}{\mathbb{R}}
\newcommand{\Z}{\mathbb{Z}}
\newcommand{\Sp}{\mathbb{S}}

\begin{document}

\title{One-cusped Dehn fillings of the sisters of the Whitehead and $6^2_2$ link complements}
\author{Priyadip Mondal}
\address{Department of Mathematics, Ben Gurion University of the Negev, David Ben-Gurion Blvd. 1, P.O.B. 653, Be'er Sheva 8410501, Israel} \email{priyadip@post.bgu.ac.il}
\begin{abstract}  In this article, we investigate the arithmeticity of the one-cusped Dehn fillings of the $(-2,3,8)$-pretzel link complement and of the Berge manifold, which respectively are the sisters of the Whitehead and $6^2_2$ link complements. We show that for each such one-cusped hyperbolic Dehn filling, the cusp field, the trace field and the invariant trace field coincide. Moreover, we establish that no one-cusped hyperbolic Dehn filling of the Berge manifold is arithmetic and that the only arithmetic one-cusped hyperbolic Dehn filling of the $(-2,3,8)$-pretzel link complement is the sister of the figure eight knot complement. The techniques used to prove these results further show that each knot complement covering a one-cusped hyperbolic Dehn filling of either of these two sisters manifolds admits no hidden symmetries, effectively generalizing already known results in this regard. \end{abstract}
\maketitle
\section{Introduction}
This paper is devoted to the study of arithmetic invariants of hyperbolic $3$-manifolds and their Dehn fillings. We here direct our attention to the $(-2,3,8)$-pretzel link complement and the Berge manifold. These are two of the four $2$-cusped orientable hyperbolic $3$-manifolds of lowest complexity $4$ (in the sense of Matveev  \cite{Matveev}) and lowest known volume, as observed by Martelli and Petronio \cite{MartPet} from the work of Callahan, Hildebrand and Weeks \cite[Appendix A]{CHW}, the other two being their “siblings” the Whitehead link and $6_2^2$ link complements respectively. Agol \cite[Theorem 3.6]{Agol} in fact proved that the complements of the Whitehead link and the $(-2,3,8)$-pretzel link do have minimal volume among all $2$-cusped such manifolds. All four of these manifolds are also arithmetic, the Whitehead link and $(-2,3,8)$-pretzel link complements covering $\mathbb{H}^3/\mathrm{PSL}_2(\mathcal{O}_1)$ and the Berge manifold and $6_2^2$ link complement covering $\mathbb{H}^3/\mathrm{PSL}_2(\mathcal{O}_3)$.

The non-hyperbolic Dehn fillings of the $(-2,3,8)$-pretzel link complement and the Berge manifold were classified in \cite{MartPet}.\footnote{In fact, the $(-2,3,8)$-pretzel link complement, Berge manifold, Whitehead link complement and $6_2^2$ link complement are exactly the exceptional examples listed in \cite[Corollary A.2 and Table A.1]{MartPet}, i.e., except these four manifolds all two cusped hyperbolic manifolds obtained by Dehn filling a single cusp of the magic manifold  have exactly five exceptional slopes for one-cusp Dehn fillings, wheres each of these four manifolds have six.} We here concentrate on the hyperbolic Dehn fillings of the $(-2,3,8)$-pretzel link complement and the Berge manifold, more specifically the one-cusped hyperbolic Dehn fillings. Our approach in this venture is three-fold: 
\begin{itemize}
\item we relate them with the notion of hidden symmetries, 
\item we classify which such ones are arithmetic, and 
\item we correlate their associated number fields. 
\end{itemize}

Recall that hidden symmetries are symmetries between finite covers which are not lifts of actual symmetries. Here, we establish the following result. 
\newcommand\hidsym{No hyperbolic knot complement covering a one-cusped Dehn filling of either the Berge manifold or the $(-2,3,8)$-pretzel link complement has hidden symmetries.}
\theoremstyle{plain}
\newtheorem*{hiddsymthm}{Theorem \ref{hidd_sym_effect}}
\begin{hiddsymthm}\hidsym\end{hiddsymthm}
The above result effectivizes \cite[Proof of Theorem 1.1]{Hoffman_3comm} and \cite[Example 5.2, Example 5.3] {CDM} (and generalizes \cite[Proof of Theorem 6.1]{Hoffman_hidden}  for orbifold fillings). Its analog was proved for the Whitehead link complement in \cite{NeumReid} and for the $6^2_2$ link complement in \cite{CDHMMW}. Our result here is effective in contrast with that the ``all but finitely many cases" conclusion in \cite[Proof of Theorem 1.1]{Hoffman_3comm} as our techniques here allow us to follow the behavior of the cusp fields of all one-cusped hyperbolic Dehn fillings of the two manifolds, whereas \cite[Proof of Theorem 1.1]{Hoffman_3comm} employed a limiting argument pertaining to orbifold fundamental groups and discerning cusp field behavior for such fillings was out of its scope. In the same vein, the method used in \cite[Example 5.2, Example 5.3] {CDM} also invokes the ``all but finitely many cases" limitation.

Theorem \ref{hidd_sym_effect} above emerges from the results and observations laid out in our second direction, which is the technical core of our paper. In this direction, we study the arithmeticity of the one-cusped Dehn fillings of the $(-2,3,8)$-pretzel link complement and the Berge manifold. The work of Borel (see \cite[Theorem 11.2.1 and Corollary 11.2.2]{MacRe}) shows that at-most finitely many such one-cusped Dehn fillings could be arithmetic. It also follows from \cite[Lemma 3.1]{Hoffman_3comm} and \cite[Theorem 2]{ReidFig8} that the family of one-cusped fillings of the Berge manifold that Hoffman considered in \cite{Hoffman_3comm} are non-arithmetic (see \cite[Proof Theorem 6.1]{Hoffman_hidden}). However, complete classification of arithmetic one-cusped Dehn fillings of the $(-2,3,8)$-pretzel link complement and the Berge manifold is un-known. In this article, we address this by proving the following two theorems. 
\newcommand\bergearith{No one-cusped hyperbolic orbifold obtained by Dehn filling a single cusp of the Berge manifold is arithmetic.}
\theoremstyle{plain}
\newtheorem*{Bergeariththm}{Theorem \ref{berge_arith_thm}}
\begin{Bergeariththm}\bergearith\end{Bergeariththm}

\newcommand\pretzelarith{The only arithmetic one-cusped hyperbolic orbifold obtained by Dehn filling a single cusp of the $(-2,3,8)$-pretzel link complement is \texttt{m003}, which is obtained by the $(1,2)$-Dehn filling on the $c_0'$ cusp with respect to $(\mathbf{m}_0, \mathbf{l}_0)$.}
\theoremstyle{plain}
\newtheorem*{Pretzelariththm}{Theorem \ref{pretzel_arith_thm}}
\begin{Pretzelariththm}\pretzelarith\end{Pretzelariththm}

Our motivation stems from the work of \cite{NeumReid} and \cite{CDHMMW} where they studied and classified the arithmetic one-cusped Dehn fillings of the Whitehead link and the $6^2_2$ link complements respectively. Our technique here is built upon the approach taken in these two papers using a number of additional subtle observations and analysis of more complex scenarios. 

The geometric data that cusp moduli captures in the results pertaining to our second direction strongly provokes a closer study of the relation of the cusp field with other closely related fields associated with our one-cusped Dehn fillings. With this in mind, we look at these one-cusped Dehn fillings from a third direction\textemdash we correlate each of their cusp field, invariant trace field and trace field. It should be noted that the first two are in fact are invariants of the commensurability class (\cite{NeumReid}, \cite{ReidTrace}). While we always have the relation 
$$\text{cusp field }\subset\text{ invariant trace field }\subset\text{ trace field}$$ 
between these number fields (see \cite[Theorem 3.1.2]{MacRe}) and \cite[Proposition 2.7]{NeumReid}), these fields are not always equal. Here, we prove the following two results. 
\newcommand\bergecusptrace{For any one-cusped hyperbolic orbifold $M_{(p,q)}$ obtained by $(p,q)$-Dehn filling on a cusp of the Berge manifold $M$, the trace field, the invariant trace field and the cusp field of $M_{(p,q)}$ coincide.}
\theoremstyle{plain}
\newtheorem*{Bergefieldthm}{Theorem \ref{Bcusptrace}}
\begin{Bergefieldthm}\bergecusptrace\end{Bergefieldthm}

\newcommand\pretzelcusptrace{For any one-cusped hyperbolic orbifold $N_{(p,q)}$ obtained by $(p,q)$-Dehn filling on a cusp of the $(-2,3,8)$-pretzel link complement, the trace field, the invariant trace field and the cusp field of $N_{(p,q)}$ coincide. }
\theoremstyle{plain}
\newtheorem*{Pretzelfieldthm}{Theorem \ref{tracecusp_pret}}
\begin{Pretzelfieldthm}\pretzelcusptrace\end{Pretzelfieldthm}
These two theorems are important for their relation with the Neumann and Reid question \cite[Question 2, \S 10]{NeumReid} asking which hyperbolic knot complements have same the trace field and the cusp field. The two theorems above show that the knot complements obtained by hyperbolic Dehn filling the unknotted component of either of the Berge manifold or the $(-2,3,8)$-pretzel link complement belong to this list. These include families of Berge knots and twisted torus knots as observed in \cite{Hoffman_3comm} and \cite{HMPT} respectively. Similar theorems have been proven for the one-cusped Dehn fillings of the Whitehead link in \cite{NeumReid} and of the $6^2_2$ link in \cite{CDHMMW}. 


Our proof methods exploit triangulations of the $(-2,3,8)$-pretzel link complement and the Berge manifold and the pertinent number theoretic properties of their associated algebraic varieties as has been done for Whitehead link in \cite{NeumReid} and for $6^2_2$ link in \cite{CDHMMW}. We should point out that the complements of the Whitehead link and the $(-2,3,8)$-pretzel link are both octahedral: they are obtained by gluing the faces of a regular ideal octahedron and hence have the volume of a regular ideal octahedron. On the other hand, the Berge manifold and the $6^2_2$-link complement are tetrahedral: they each are made of four regular ideal tetrahedra and consequently have volume $4v_0$, where $v_0$ is the volume of a regular ideal tetrahedron.

To prove Theorem \ref{berge_arith_thm} and \ref{pretzel_arith_thm}, we use the characterization of non-compact arithmetic Kleinian group of finite co-volume in \cite[Proposition 4.4]{NeumReid}, which together with \cite[Proposition 2.7]{NeumReid} would require the cusp field of a one-cusped arithmetic hyperbolic orbifold to be quadratic imaginary. We utilize the ideal triangulations of our two manifolds and the fact that their one-cusped hyperbolic Dehn fillings can be realized as points in a co-dimension one subvariety $\mathcal{V}_0$ of the gluing variety $\mathcal{V}$ associated with these triangulations. It is important to bear in mind that \cite[Proposition 1.6]{CDM} implies that the cusp field of such a one-cusped Dehn filling is generated over $\mathbb{Q}$ by the value of a rational function on $\mathcal{V}$ at the corresponding point in $\mathcal{V}_0$. 

We point out that both the $(-2,3,8)$-pretzel link complement and the Berge manifold have non-trivial self-isometries, which can be checked via SnapPy \cite{snappy}. This helps us conclude that analyzing just the one-cusped $(p,q)$-Dehn fillings, where both $p$ and $q$ are non-negative integers, is enough. We first analyze such one-cusped Dehn fillings for which the computed cusp moduli is a unit non-rational algebraic integer. For the Berge manifold, all of them satisfy this (Proposition \ref{Balgint}) and for the $(-2,3,8)$-pretzel link complement it is true when the filling coefficients satisfy $p\ne 2q$ (Proposition \ref{pret_algint}). We show that the cusp fields of the $(2q,q)$ one-cusped fillings of the $(-2,3,8)$-pretzel link complement are not quadratic imaginary. As observed in \cite{CDHMMW} (recorded as Fact \ref{quad_unit} in Section \ref{onecusp622}), there are only three quadratic imaginary number fields which have a unit non-rational algebraic integer. We argue that none of these could be the cusp fields of the respective one-cusped Dehn fillings of the Berge manifold with unit algebraic integral cusp moduli and the sister \texttt{m003} of the figure eight knot complement is the only one-cusped Dehn filling of the $(-2,3,8)$-pretzel link complement with such a cusp field. 



Theorem \ref{hidd_sym_effect} can derived as a consequence of \cite[Proposition 9.1]{NeumReid} and the fact that the sister of the figure eight knot complement is the only one amongst the considered one-cusped hyperbolic Dehn fillings which has a quadratic imaginary cusp field (proof of Theorem \ref{berge_arith_thm} and \ref{pretzel_arith_thm}). This method of using degree bound for the cusp field gives us stronger control on understanding which one-cusped Dehn fillings can cover rigid cusped orbifolds\footnote{in relation with \cite[Proposition 9.1]{NeumReid}} than the argument presented the proof of \cite[Theorem 1.1]{Hoffman_3comm}, which involves the limiting behavior of the sequence of the orbifold fundamental groups of these fillings and hence does not give us an effective conclusion for such fillings of the Berge manifold. On the other hand, proof of \cite[Theorem 6.1]{Hoffman_hidden} used degrees of covers to one-rigid cusped orbifolds and is only applicable to one-cusped manifold fillings of the Berge manifold, and both of \cite[Example 5.2 and 5.3]{CDM} used the non-geometric isolation of cusps, which too only capture the limiting behavior of the one-cusped fillings of both of our manifolds and fail to be effective. 

The proof techniques for Theorem \ref{Bcusptrace} and \ref{tracecusp_pret} use Poincar\'e polyhedron theorem as in \cite[Theorem 6.2]{NeumReid} and \cite[Proposition 6.16]{CDHMMW}. They are also reliant on the data provided by the ideal triangulations of the two manifolds.


To make the paper self-contained, we devote Section \ref{background} and \ref{Whitehead and 622 review} of this paper for literature review. In Section \ref{background}, we recall the background on notions associated with cusps, ideal triangulation, Dehn filling, and gluing variety. In Section \ref{Whitehead and 622 review}, we review the one-cusped Dehn fillings of the Whitehead link complement and the $6^2_2$-link complement from \cite{NeumReid} and \cite{CDHMMW} respectively. 

From Section \ref{pretzel berge gluing} onwards, we develop and derive the tools and findings to establish the main results of this paper. In Section \ref{pretzel berge gluing}, we compute the gluing variety $\mathcal{V}$ and its co-dimension one subvariety $\mathcal{V}_0$ for both the Berge manifold and the $(-2,3,8)$-pretzel link complement. In Section \ref{Berge V0 parametrization}, we describe a one variable parametrization of $\mathcal{V}_0$ for the Berge manifold and present a proof of Theorem \ref{Bcusptrace}. Section \ref{arithmetic berge} consists of analysis of the one-cusped Dehn fillings of the Berge manifold and a proof of Theorem \ref{berge_arith_thm}. In Section \ref{pretzel V0 parametrization}, we show that $\mathcal{V}_0$ for the $(-2,3,8)$-pretzel link complement can be parametrized by cusp parameter $\tau$ and prove Theorem \ref{tracecusp_pret}. Section \ref{pretzel arithmetic} is devoted to study the one-cusped Dehn fillings of the $(-2,3,8)$-pretzel link complement and prove Theorem \ref{pretzel_arith_thm}. In Section \ref{hidd_symm}, we lay out the argument for Theorem \ref{hidd_sym_effect}. 
\subsection*{Acknowledgements}
The author would like to thank Jason DeBlois and Neil Hoffman for providing many helpful feedback and comments on earlier drafts of the article. The author was partially supported by Israeli Science Foundation grant 823/23 (through the Postdoctoral fellowship at Ben Gurion University). The author also acknowledges the supporting work environment in the math department of Rutgers University-New Brunswick where some of the work was completed. 

\section{Background}\label{background}
In this section, we recall the necessary background materials from the literature and set up the basic notations that we will use throughout the paper. 
\subsection{Cusps of hyperbolic $3$-orbifolds}
Unless otherwise specified, by a hyperbolic $3$-orbifold, we will mean an orbifold $\mathbb{H}^3/\Gamma$ where $\Gamma$ is a discrete subgroup of $\operatorname{Isom}^{+}(\mathbb{H}^3)=\operatorname{PSL}_{2}(\C)$ of finite co-volume. In particular, $\mathbb{H}^3/\Gamma$ whose (orbifold) fundamental group is $\Gamma$ inherits a complete hyperbolic metric from $\mathbb{H}^3$. Denote the hyperbolic $3$-orbifold $\mathbb{H}^3/\Gamma$ by $M_{\Gamma}$. The $\Gamma$-orbits of the parabolic fixed points of $\Gamma$ on the extended complex plane $\widehat{\C}$ are referred to as the \textit{cusps} of $M_{\Gamma}$. For a parabolic fixed point $\mathbf{x}$ of $\Gamma$, let $c=\Gamma \mathbf{x}$ denote the corresponding cusp of $M_{\Gamma}$. Let $\Lambda_{\mathbf{x}}$ be the stabilizer of $\mathbf{x}$ in $\Gamma$. If $B_{\mathbf{x}}$ is a horoball centered at $\mathbf{x}$, then $\Lambda_{\mathbf{x}}$ acts as an oriented wallpaper group of Euclidean isometries on the horosphere $\partial B_{\mathbf{x}}\cong \R^2$. It follows from \cite[Corollary 2.2]{DuMe} that $B_{\mathbf{x}}/\Lambda_{\mathbf{x}}$ embeds into $M_{\Gamma}$ and is referred to as a \textit{cusp neighborhood} of cusp $c$ (also see \cite[Corollary 5.10.2]{Th_notes}). We say that $\partial B_{\mathbf{x}}/\Lambda_{\mathbf{x}}$ is a \textit{cusp cross-section} of cusp $c$. This cusp cross-section is one of the five oriented Euclidean $2$-orbifolds: $\mathbb{T}^2$, $\mathbb{S}^2(2,2,2,2)$, $\mathbb{S}^2(2,3,6)$, $\mathbb{S}^2(3,3,3)$, $\mathbb{S}^2(2,4,4)$. When $\Lambda_{\mathbf{x}}$ is $\mathbb{Z}^2$, this cusp cross-section is a Euclidean torus. In this case, we say that $c$ is a torus cusp.  When $\Gamma$ does not have any torsion elements, $M_{\Gamma}$ is a manifold and all of its cusps are torus cusps. 

Assume that $c$ is a torus cusp. If $(\mathbf{m}, \mathbf{l})$ is a generating pair for $\Lambda_{\mathbf{x}}$, then we say that $(\mathbf{m}, \mathbf{l})$ is an \textit{oriented pair} if the determinant of the $2\times2$ matrix $\begin{pmatrix} \mathbf{m} \quad \mathbf{l} \end{pmatrix}$ is positive. For an oriented pair of generators $(\mathbf{m}, \mathbf{l})$ for $\Lambda_{\mathbf{x}}$, the complex number $\frac{\mathbf{l}}{\mathbf{m}}$ belongs to the upper-half plane and is referred to as a \textit{cusp moduli} of cusp $c$. It is easy to see that $\operatorname{PSL}_2(\mathbb{Z})$ acts transitively on the set of cusp moduli of cusp $c$ and consequently, the field $\Q(\frac{\mathbf{l}}{\mathbf{m}})$ is independent of the choice of the oriented generating pair $(\mathbf{m}, \mathbf{l})$. This field is called the \textit{cusp field} of cusp $c$. 

\subsection{Ideal triangulation and associated algebraic varieties}
In this section, we will review the notions of ideal triangulation, gluing variety and the completeness equation. Contents of this section stem from the work of \cite[Chapter 4]{Th_notes} and \cite[Section 2]{NZ}. Other good references for these contents are \cite[Chapter E]{BePe} and \cite[Chapter 10]{Ratcliffe}. 
\subsubsection{Ideal triangulation and gluing variety}
An \textit{ideal tetrahedron} $T$ in $\mathbb{H}^3$ is a hyperbolic tetrahedron in $\mathbb{H}^3 \cup \partial_{\infty} \mathbb{H}^3$ such that only the vertices of $T$ are in $\partial_{\infty} \mathbb{H}^3=\widehat{\C}$. If these vertices, referred to as the \textit{ideal vertices} of $T$, are $\mathbf{v_0}$, $\mathbf{v_1}$, $\mathbf{v_2}$ and $\mathbf{v_3}$, such that the cross-ratio $\mathbf{z}=[\mathbf{v_0}, \mathbf{v_1}, \mathbf{v_2}, \mathbf{v_3}]=\frac{\mathbf{v_3}-\mathbf{v_0}}{\mathbf{v_2}-\mathbf{v_0}} \frac{\mathbf{v_2}-\mathbf{v_1}}{\mathbf{v_3}-\mathbf{v_1}}$ is in the upper half plane, then there exists an isometry $\phi$ of $\mathbb{H}^3$ which sends $\mathbf{v_0}$ to $0$, $\mathbf{v_1}$ to $\infty$, $\mathbf{v_2}$ to $1$, and $\mathbf{v_3}$ to $\mathbf{z}$. This $\mathbf{z}$ is called the \textit{shape parameter} of $T$ corresponding to the edge joining $\mathbf{v_0}$ and $\mathbf{v_1}$. Note that the shape parameter of the edge joining $\mathbf{v_2}$ and $\mathbf{v_1}$ (respectively, $\mathbf{v_3}$ and $\mathbf{v_1}$) is $\frac{1}{1-\mathbf{z}}$ (respectively, $\frac{\mathbf{z}-1}{\mathbf{z}}$). One can also check that the shape parameters of $T$ corresponding to two opposite edges (i.e., edges that do not share an ideal vertex) are equal. 
\begin{remark}
In rest of the paper, we will use the notations $\zeta_1(\mathbf{z})$ and $\zeta_2(\mathbf{z})$ from \cite{CDM} to denote $\frac{1}{1-\mathbf{z}}$ and $\frac{\mathbf{z}-1}{\mathbf{z}}$ respectively. Note, all of $\mathbf{z}$, $\zeta_1(\mathbf{z})$ and $\zeta_2(\mathbf{z})$ are in the upper-half plane. In particular, we do not allow \textit{flat tetrahedron} in our definition. So, the shape parameters are never $0$ or $1$. 
\end{remark}

An \textit{ideal triangulation} of a hyperbolic $3$-manifold with cusps $M_{\Gamma}=\mathbb{H}^3/\Gamma$ is a pair $(\mathcal{T}, \mathcal{F})$ where $\mathcal{T}$ is a set of $n$-distinct ideal tetrahedra $T_0, \dots, T_{n-1}$ and $\mathcal{F}$ is a set of $2n$ pairs of isometries of $\mathbb{H}^3$ satisfying the following:
\begin{enumerate}
\item for each $(f,g)\in \mathcal{F}$, $g=f^{-1}$,
\item for each $(f,g)$, there exists a unique set of two distinct faces $\{F_{fg}, F_{gf}\}$ from these tetrahedra in $\mathcal{T}$ and vice versa such that $f(F_{fg})=F_{gf}$ and $g(F_{gf})=F_{fg}$, (the faces $F_{fg}$ and $F_{gf}$ can belong to two distinct tetrahedra from $\mathcal{T}$),
\item for each $(f,g)\in \mathcal{F}$, if $T_{fg}$ and $T_{gf}$ are the tetrahedra containing respectively $F_{fg}$ and $F_{gf}$, then, $f(T_{fg})\cap T_{gf}=F_{gf}$ and $g(T_{gf})\cap T_{fg}=F_{fg}$, and, 
\item if $\rho_{\mathcal{F}}$ is the equivalence relation on $\bigsqcup\limits_{i=0}^{n-1} T_i$ induced by $\mathcal{F}$, then $M_\Gamma$ is homeomorphic to the quotient space $\bigsqcup\limits_{i=0}^{n-1} T_i/\rho_{\mathcal{F}}$. 
\end{enumerate}
See \cite[Section 10.1]{Ratcliffe} and \cite[Section E.5-i,-ii]{BePe} for more details. An ideal triangulation $(\mathcal{T}, \mathcal{F})$ on $M_{\Gamma}$ also induces a simplicial structure on $M_{\Gamma}$ with simplices of dimension $1$, $2$ and $3$, referred to as respectively edges, faces and tetrahedra of $M_{\Gamma}$. Note that $M_{\Gamma}$ will have $n$ edges, $2n$ faces and $n$ tetrahedra. An edge $E$ of $M_{\Gamma}$ is an equivalence class of edges of $\mathcal{T}$. Since $M_{\Gamma}$ is hyperbolic, the total angle around an edge $E$ of $M_{\Gamma}$ should add up to $2\pi$ (see \cite[Equation 4.2.2]{Th_notes}) and thus giving the following \textit{edge equation} for $E$ (see \cite[Equation 4.2.1]{Th_notes}): 
\begin{equation*}
\prod_{e \in E} \mathbf{z}_{e}=1, \text{ where $\mathbf{z}_{e}$ is the shape parameter corresponding to edge $e$ from $\mathcal{T}$}.
\end{equation*}
So, the ideal triangulation $(\mathcal{T}, \mathcal{F})$ on $M_{\Gamma}$ gives us $n$ edge equations. For each $i \in \{0, \dots, n-1\}$, fix an edge for $T_i$ and denote its shape parameter by $\mathbf{z_i}$. Since the other shape parameters of $T_i$ are M\"obius transformations of $\mathbf{z_i}$, each edge equation will be equivalent to a polynomial equation in $\mathbf{z_0}, \dots, \mathbf{z_{n-1}}$. Replacing $\mathbf{z_i}$ by variable $z_i$ for each $i \in \{0, \dots, n-1\}$, we get $n$ such polynomial equations in $n$ variables $z_0, \dots, z_{n-1}$. The algebraic set consisting of all $n$-tuples $(z_0, z_1, \dots, z_{n-1})$ in $(\mathbb{C}-\{0,1\})^n$ which satisfy these polynomial equations is called the \textit{gluing variety} associated with $(\mathcal{T}, \mathcal{F})$. We will denote this gluing variety by $\mathcal{V}(M_{\Gamma})$. 

\subsubsection{Cusp triangulation and holonomy equations}
Consider $M_{\Gamma}$ and its ideal triangulation $(\mathcal{T}, \mathcal{F})$ as in the previous subsection. Let $c_0, c_1, \dots, c_{k-1}$ be the cusps of $M_{\Gamma}$. Denote a torus cusp cross-section for cusp $c_j$ by $\mathbb{T}_j$. Then $(\mathcal{T}, \mathcal{F})$ induces a triangulation of $\mathbb{T}_j$, referred to as a \textit{cusp triangulation} of $\mathbb{T}_j$. The triangles in this cusp triangulation of $\mathbb{T}_j$ are cross-sections of the tetrahedra in $\mathcal{T}$ by horospheres centered at the ideal vertices of those tetrahedra.

We will now closely follow \cite[Section 2]{NZ} to define the notion of holonomy derivative and holonomy equation.  Suppose an element in the first homology of $\mathbb{T}_j$ is represented by a simplicial curve $\gamma$. Then $\gamma$ consists of finite number of $1$-simplices. We refer to this number as the \textit{simplicial length} of $\gamma$ and denote it by  $l(\gamma)$. One defines 
$$\mu([\gamma])=(-1)^{l(\gamma)} \prod_{e\in E_{\gamma}} \mathbf{z}_{e},$$
where $E_{\gamma}$ is the set of edges from $\mathcal{T}$ which $\gamma$ meets from the right. It was shown in \cite[Lemma 2.1]{NZ} that this map $\mu: \operatorname{H}_1(\mathbb{T}_j)\to \C^{\ast}$ is well-defined and a group homomorphism. $\mu([\gamma])$ is referred to as the \textit{holonomy derivative} of $[\gamma]$. 

Since $M_{\Gamma}$ is complete, all the torus cusp cross-sections of $M_{\Gamma}$ have Euclidean structures (see \cite[Proposition E.6.5]{BePe}). This implies that for all cusps of $M_{\Gamma}$, all the above mentioned holonomy derivatives are $1$. In particular, if $(\mathbf{m_j}, \mathbf{l_j})$ is an oriented generating pair for $\operatorname{H}_1(\mathbb{T}_j)$, then 
\begin{equation}\label{holo_j}
\mu(\mathbf{m_j})=\mu(\mathbf{l_j})=1.
\end{equation} 
(See \cite[Proposition E.6.12]{BePe}). Equation \ref{holo_j} is referred to as a \textit{holonomy equation} for cusp $c_j$. As in the case of edge equations, each of the $k$ holonomy equations for $M_{\Gamma}$ can be turned into a polynomial equation in $z_0, z_1, \dots, z_{n-1}$. In this case, the set of points in the gluing variety which further satisfy all the holonomy polynomial equations consists of isolated points (see \cite[Proposition E.6.16]{BePe}) and contain the actual $n$-tuple of shape parameters $(\mathbf{z_0}, \dots, \mathbf{z_{n-1}})$ for the tetrahedra $T_0, \dots, T_{n-1}$. We will refer to the point  $(\mathbf{z_0}, \dots, \mathbf{z_{n-1}})$ in the gluing variety as the \textit{complete structure}. 
\subsubsection{Dehn fillings and incomplete structures}
Choose an oriented generating pair $(\mathbf{m_j}, \mathbf{l_j})$ of the first integral homology $\operatorname{H}_1(\mathbb{T}_j)$ of the torus cusp cross-section for cusp $c_j$. Let $N_j$ be a cusp neighborhood of cusp $c_j$ cut off by $\mathbb{T}_j$ (i.e. $N_j \cong \mathbb{T}^2\times [0, \infty)$). We recall that for co-prime integers $p$ and $q$, the $(p,q)$-\textit{Dehn filling} of $c_j$ with respect to $(\mathbf{m_j}, \mathbf{l_j})$ is the manifold 
$$\left(M_{\Gamma}-int(N_j)\right)\cup_{h} \left(\mathbb{D}^2\times \mathbb{S}^1\right)$$
where $h$ is a self-homeomorphism of the torus that sends its meridian to the curve $p\mathbf{m_j}+q\mathbf{l_j}$ of $\partial N_j=\mathbb{T}_j$. Denote this $(p,q)$-Dehn filling with respect to $(\mathbf{m_j}, \mathbf{l_j})$ as $M_{\Gamma, (p,q)}$. One can also talk about $(p,q)$-Dehn filing when $p$ and $q$ are not co-prime, i.e, $gcd(p,q)=d>1$. In this case, the process is similar but instead of attaching a solid tours $\mathbb{D}^2\times \mathbb{S}^1$, we attach a (solid) orbi-torus $\left(\mathbb{D}^2\times \mathbb{S}^1\right)_d$ of order $d$. Here, $\left(\mathbb{D}^2\times \mathbb{S}^1\right)_d$ is the quotient of a solid torus by an order $d$ rotation along its core curve. Needless to say, $M_{\Gamma, (p,q)}$ is an orbifold where the core curve of the attached (solid) orbi-torus is a singular locus of order $d$. See \cite[Figure 1]{CDHMMW} for a picture. Here, $(p,q)$ is referred to as the \textit{Dehn filling co-efficient} for $M_{\Gamma, (p,q)}$. 

One can Dehn fill as many cusps as they want. If a cusp is not Dehn filled, we say that the corresponding Dehn filling co-efficient is $\infty$. It follows from Thurston's Dehn surgery theorem \cite[Theorem 5.8.2]{Th_notes}, which was extended for the orbifold case by \cite[Theorem 5.3]{DuMe}, that when the $k$-tuple consisting of Dehn filling coefficients for all $k$ cusps lie outside a compact set in $(\mathbb{R}^2)^k$, the resulting orbifold obtained by Dehn filling is hyperbolic. 

The triangulation corresponding to any $\vec{z}=(z_0, z_1, \dots, z_{n-1})$ in $\mathcal{V}(M_{\Gamma})$ with each $z_j$ in the upper half-plane gives a hyperbolic structure on the topological manifold corresponding to $M_{\Gamma}$. Let's denote this hyperbolic structure by $M_{\vec{z}}$. However, these $M_{\vec{z}}$'s are, in general, incomplete since they might not satisfy one or more holonomy equations. For such an $n$-tuple $\vec{z}$ in $\mathcal{V}(M_{\Gamma})$ with each  $z_j$ in the upper-half plane, consider the logarithms of the corresponding holonomy derivatives $u_j=\operatorname{log}\left(\mu(\mathbf{m_j})\right)$ and $v_j=\operatorname{log}\left(\mu(\mathbf{l_j})\right)$. The real numbers $p_j$ and $q_j$ which satisfy the equations $p_j u_j +q_j v_j=2 \pi \mathrm{i}$ are referred to as \textit{generalized Dehn surgery coefficients} (see \cite[Definition 4.5.1]{Th_notes}, \cite[Lemma 4.1 and Equation 32]{NZ}). It further follows from the argument given by Thurston that if $p_j$ and $q_j$ are integers so that the metric completion $\widehat{M_{\vec{z}}}$ of $M_{\vec{z}}$ is hyperbolic then, $\widehat{M_{\vec{z}}}$ is isometric to the orbifold obtained from $M_{\Gamma}$ by $(p_j,q_j)$-Dehn filling the $c_j$ cusp for all $j \in \{0, \dots, k-1\}$ (see \cite[Proposition E.6.27]{BePe}). Note that when $p_j$ and $q_j$ are integers, the above equation in $u_j$ and $v_i$ implies that for each such $j$
\begin{equation}\label{Dehneq_j}
\mu(\mathbf{m_j}^{p_j} \mathbf{l_j}^{q_j})=1.
\end{equation}
This equation is called the $j$-th \textit{Dehn filling equation}. 
It is worth mentioning that if for $\vec{z}$ in $\mathcal{V}(M_{\Gamma})$ with each $z_j$ in the upper-half plane, $\widehat{M_{\vec{z}}}$ is obtained by Dehn filling a few cusps of $M_{\Gamma}$, then \cite[Proposition 1.6]{CDM} shows that for each cusp $\widehat{c_j}$ of $\widehat{M_{\vec{z}}}$ corresponding to an unfilled cusp $c_j$ of $M_{\Gamma}$, there is a rational function $\tau_j$ on $\mathcal{V}(M_{\Gamma})$ such that $\tau_j(\vec{z})$ is a cusp moduli of $\widehat{M_{\vec{z}}}$ for the $\widehat{c_j}$-cusp. 

\section{One cusp Dehn fillings of a two cusped hyperbolic $3$-manifold}\label{Whitehead and 622 review}
Let $M_{\Gamma}$ denote a hyperbolic $3$-manifold, $c_j$ its cusps for $j \in \{0, \dots, k-1\}$, and $(\mathcal{T}, \mathcal{F})$ its ideal triangulation as in the previous section. In this section, we will consider the case when $k=2$, i.e., when $M_{\Gamma}$ has two cusps. Suppose we only Dehn fill the $c_1$ cusp and keep the $c_0$ cusp complete, i.e. the holonomy equation \ref{holo_j} is satisfied only for cusp $c_0$. Consider the subvariety $\mathcal{V}_0(M_{\Gamma})$ of $\mathcal{V}(M_{\Gamma})$ consisting of tuples $\vec{z}=(z_0, z_1,\dots, z_{n-1})\in \mathcal{V}(M_{\Gamma})$ which also satisfy equation \ref{holo_j} when $j=0$. Note that $\mathcal{V}_0(M_{\Gamma})$ is a one dimensional subvariety of $\mathcal{V}(M_{\Gamma})$ with co-dimension $1$. Denote the hyperbolic orbifold obtained by hyperbolic $(p,q)$-Dehn filing the $c_1$-cusp of $M_{\Gamma}$ as  $M_{\Gamma, (p,q)}$. A tuple $\vec{z}=(z_0, z_1,\dots, z_{n-1})\in \mathcal{V}_0(M_{\Gamma})$ that corresponds to $M_{\Gamma, (p,q)}$ can be obtained by solving the Dehn filling equation \ref{Dehneq_j} for $j=1$. 


\subsection{One-cusped Dehn fillings of the Whitehead link complement from \cite{NeumReid}}
Let's denote the Whitehead link complement by $M_W$. Neumann and Reid \cite{NeumReid} worked with an ideal triangulation of $M_W$ consisting four ideal tetrahedra, each with an edge with shape parameter $\mathrm{i}$. They denoted these four shape parameters as $x$, $y$, $z$ and $w$, which we will denote here respectively as $\mathbf{z_0}$, $\mathbf{z_1}$, $\mathbf{z_2}$ and $\mathbf{z_3}$ to keep the notational consistency. It was shown there that $\mathcal{V}(M_W)$ equals $$\left \{(z_0, z_1, z_2, z_3) \in \left(\C-\{0,1\}\right)^4: z_0 z_1 z_2 z_3=1, (1-z_0)(1-z_3)=(1-z_1)(1-z_2)\right \},$$ whereas $\mathcal{V}_0(M_W)=\left \{(z_0, -z_0^{-1}, z_0, -z_0^{-1}): z_0 \in \mathbb{C}-\{0,1\}\right \}$. It can further be seen from \cite[Proof of Theorem 6.3]{NeumReid} that for a $(p,q)$-Dehn filling on a single cusp of $M_W$ if $\vec{z}=(z_0, z_1, z_2, z_3)$ in $\mathcal{V}_0(M_W)$ is the corresponding incomplete structure, then $z_0$ is a unit algebraic integer when $p+4q\ne0$. 
\subsection{One-cusped Dehn fillings of the  $6^2_2$ link complement from \cite{CDHMMW}}\label{onecusp622}
Let's denote the $6^2_2$ link complement by $M_{6^2_2}$. The ideal triangulation of $M_{6^2_2}$ that \cite{CDHMMW} considered for computing varieties $\mathcal{V}$ and $\mathcal{V}_0$ has four regular ideal tetrahedra. The notation for the four shape parameters, one for each ideal tetrahedra, in \cite{CDHMMW} are $x, y, z$ and $w$. We will denote them here respectively as $\mathbf{z_0}$, $\mathbf{z_1}$, $\mathbf{z_2}$ and $\mathbf{z_3}$. Note that all these $\mathbf{z_j}$'s and their images under $\zeta_1$ and $\zeta_2$ are $\frac{1+\sqrt{3}\mathrm{i}}{2}$. It was shown in \cite{CDHMMW} that
\begin{align*}
\mathcal{V}(M_{6^2_2})&=\left \{(z_0, z_1, z_2, z_3)\in \left (\C-\{0,1\}\right)^4: (1-z_0)(1-z_1)(1-z_2)(1-z_3) = z_0z_3 = z_1z_2\right \},\text{ and}\\ 
\mathcal{V}_0(M_{6^2_2})&=\left \{(z_0, z_3) \in \left (\C-\{0,1\}\right)^2: 1-z_0(1-z_0)z_3(1-z_3)=0\right\}.
\end{align*} 
Unlike the Whitehead case, $\mathcal{V}_0(M_{6^2_2})$ is not parametrized by a single variable corresponding to a shape parameter. Thus, the Dehn filling points in $\mathcal{V}_0(M_{6^2_2})$ were studied in \cite{CDHMMW} using the symmetric homogeneous polynomial $s_d(u,v)=\sum_{j=0}^{d}u^j v^{d-j}$, and the regular map $\phi: \mathbb{C}^2 \to \mathbb{C}^2$ defined as $\phi(u,v)=(u+v, uv)$. As noted in \cite[Equation 4]{CDHMMW}, $u^{d+1}-v^{d+1}=(u-v) s_d(u,v)$. 

If $u+v$ and $uv$ are denoted by $\sigma_1$ and $\sigma_2$, \cite[Lemma 6.5]{CDHMMW} shows that there is polynomial $t_d \in \mathbb{Z}[\sigma_1,\sigma_2]$ such that $\phi^{\ast}(t_d)=s_d$ where $\phi^{\ast}: \mathbb{Z}[\sigma_1,\sigma_2] \to \mathbb{Z}[u,v]$ is the pullback ring homomorphism sending $\sigma_1$ and $\sigma_2$ to $u+v$ and $uv$ respectively, and,

\begin{equation*}
t_d(\sigma_1,\sigma_2)=\sigma_1^d+(-d+1) \sigma_1^{d-2} \sigma_2+n_2 \sigma_1^{d-4} \sigma_2^2+ \ldots+ n_{\lfloor d/2 \rfloor} \sigma_1^{\delta} \sigma_2^{\lfloor d/2 \rfloor}
\end{equation*}
where all $n_j$'s are integers, $\delta \in \{0,1\}$ and $\delta=d \operatorname{mod} 2$.

It is remarked in \cite[p. 5338]{CDHMMW} that $\phi^{\ast}(1+\sigma_1 \sigma_2 -\sigma_2-\sigma_2^2)=1-z_0(1-z_0)z_3(1-z_3)$ and $\phi$ restricts to a regular map from $\mathcal{V}_0(M_{6^2_2})$ to the zero set $\mathcal{U}_0(M_{6^2_2})$ of $1+\sigma_1 \sigma_2 -\sigma_2-\sigma_2^2$. Note that $\mathcal{U}_0(M_{6^2_2})$ can be parametrized just by $\sigma_2$. Let $(p,q)$ be a pair of non-negative integers and $(\sigma_1, \sigma_2)$ be the corresponding point in $\mathcal{U}_0(M_{6^2_2})$ for the hyperbolic $(p,q)$-Dehn filling, then was shown in \cite[Proposition 6.6]{CDHMMW} that when $q \ne 0$ and $p \ne 2q$, $\sigma_2$ is a unit algebraic integer.  We end this section by recording the following observation made in \cite[Proof of Corollary 6.14]{CDHMMW} as a fact below, which we will use later. 

\begin{fact}\label{quad_unit}
If $\mathbf{z}$ is a unit non-rational algebraic integer with positive imaginary part such that $\Q(\mathbf{z})$ is quadratic imaginary, then $\Q(\mathbf{z})$ is either $\Q(\mathrm{i})$ (and $\mathbf{z}=\mathrm{i}$) or $\Q(\mathrm{i}\sqrt{3})$ (and $\mathbf{z}=\frac{\pm 1+\sqrt{3}\mathrm{i}}{2}$). 
\end{fact}



\section{Gluing varieties of the Berge manifold and the $(-2,3,8)$-pretzel link complement}\label{pretzel berge gluing}
In this section, we describe gluing variety $\mathcal{V}$ and its subvariety $\mathcal{V}_0$ for both the Berge manifold and the $(-2,3,8)$-pretzel link complement. The links associated with these two sister manifolds are shown in Figure \ref{link_pics}. We will use SnapPy \cite{snappy} to procure the associated triangulations for these link complements. We first record the following fact which we will use multiple times in the paper. 

\begin{fact}\label{sym_fact}
For both the Berge manifold and the $(-2,3,8)$ pretzel link complement, it can be checked from SnapPy \cite{snappy} that the cusps are exchanged by an orientation preserving isometry of the manifold. So, for each of the link complements shown in Figure \ref{link_pics}, a one-cusped hyperbolic Dehn filling of the knotted cusp is isometric to some one-cusped hyperbolic Dehn filling of the unknotted cusp and vice versa (see \cite[Fact 2.3]{Mon}).
\end{fact}

We now make a remark on the notation that we will use in all cusp triangulation pictures in this paper. 
\begin{remark}
For a cusp triangle in a cusp triangulation, $T^j_i$ represents the view from the $j$-th (ideal) vertex of the (ideal) tetrahedron $T_i$ and the label associated with one of its vertex represents the label of the other (ideal) vertex of the corresponding tetrahedral edge incident to the (ideal) vertex $j$. 
\end{remark}

Before we begin, we briefly recall the notions of isomorphism signature and isometry signature from \cite{Burton} and \cite{FGGTV} respectively. 
\subsection{Isomorphism signature and Isometry signature} Defined by Burton \cite[\S 3.2, p. 157]{Burton}, the isomorphism signature of a three-manifold triangulation is a string uniquely representing the gluing data of the said triangulation. Burton \cite[Theorem 6]{Burton} also showed that the isomorphism signature is a complete invariant of the triangulation isomorphism class. The isometry signature of a hyperbolic three-manifold is a complete invariant of its isometry class, defined in \cite[Definition 3.4]{FGGTV} as one of its specific isomorphism signature related to its Epstein-Penner canonical cell decomposition \cite{EpPe}.

\begin{figure}
\includegraphics[scale=.396]{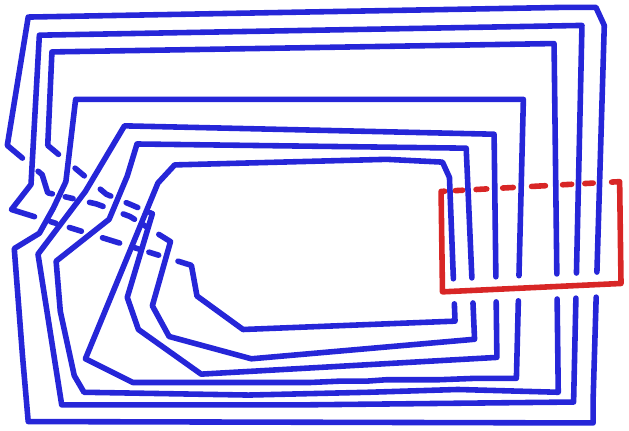} \qquad \includegraphics[scale=.324] {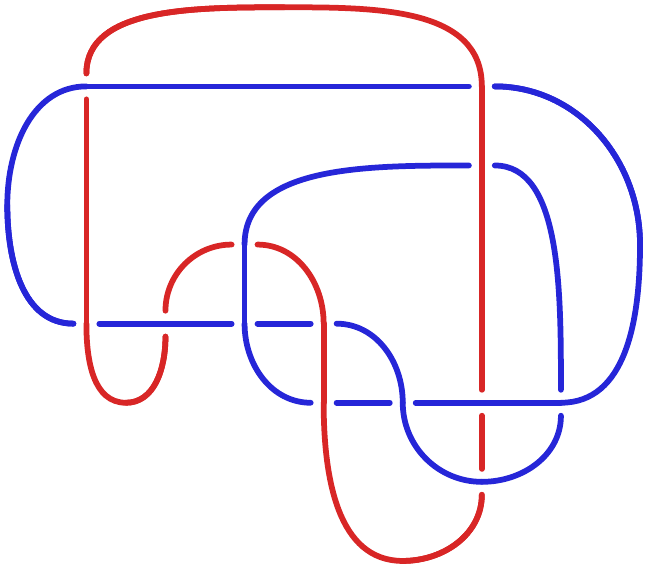}
\caption{Left: The Berge manifold link (picture drawn on SnapPy \cite{snappy}), Right: Link \texttt{L13n5885} (picture taken from SnapPy\cite{snappy})}
\label{link_pics}
\end{figure}
\subsection{Berge Manifold}\label{berge_gluing}


%

We will use $M$ denote the Berge manifold.  On SnapPy \cite{snappy}, we first draw the link whose complement is the Berge manifold. This link is shown in the left picture of Figure \ref{link_pics} (cf. picture in \cite[row $3$ of Table A.1]{MartPet}). Then we retrieve the SnapPy triangulation from this by appending the \texttt{.\_get\_tetrahedra\_gluing\_data()} command on this SnapPy manifold, which turns out to be a regular ideal triangulation (appending the \texttt{.tetrahedra\_shapes()} command will show this). The triangulation isomorphism signature for this triangulation is
$$\texttt{eLMkbbdddemdxi\_baBddIaBBrt}.$$

Let $T_0, T_1, T_2$ and $T_3$ to denote the four tetrahedra in this regular ideal triangulation of $M$. The (ideal) vertices of these (ideal) tetrahedra are labelled as $0$, $1$, $2$ and $3$ in SnapPy \cite{snappy}. Let $z_0$ (respectively, $z_2$) denote the shape parameter of the edge of $T_0$ (respectively, $T_2$) joining ideal vertices $0$ to $2$, $z_1$ denote the shape parameter of the edge of $T_1$ joining ideal vertices $0$ and $3$, and  $z_3$ denote the shape parameter of the edge of $T_3$ joining ideal vertices $0$ and $1$.
Then the edge equations turn out to be 
\begin{align}
z_0 z_1 &=\zeta_2(z_2) \zeta_2(z_3)\label{Be0}\\
z_2 z_3 &=\zeta_2(z_0) \zeta_2(z_1). \label{Be1}
\end{align}
So, the gluing variety $\mathcal{V}(M)$ of the Berge manifold is the set of $4$-tuples $(z_0,z_1,z_2,z_3)$ in $\left(\mathbb{C}-\{0,1\}\right)^4$ cut out by equations \ref{Be0} and \ref{Be1}. 

\begin{figure}
\begin{tikzpicture}[scale=1.75]

\draw [red, thick, ->>](.5,.866)--(.75, .433);
\draw [red, thick](.75,.433)--(1, 0);

\node at (1.5, .3) {$T_0^0$};
\node at (1.2, .15){$1$};
\node at (1.5, .7){$2$};
\node at (1.8, .15){$3$};

\node at (1, .6) {$T_1^0$};
\node at (.7, .73){$1$};
\node at (1.3, .73){$2$};
\node at (1, .2){$3$};
\draw [blue, thick, ->>](1,0)--(1.5,0);
\draw [blue, thick](1.5,0)--(2,0);
\draw [blue, thick, ->>](.5,.866)--(1,.866);
\draw [blue, thick](1,.866)--(1.5,.866);

\draw (1,0)--(1.5,.866);
\draw [red, thick, ->>](1.5,.866)--(1.75, .433);
\draw [red, thick](1.75,.433)--(2, 0);
\draw[blue] (.5,.866)--(1.5,.866);


\end{tikzpicture}
\caption{Cusp triangulation of the unknotted cusp of the Berge manifold}
\label{cusp0_Berge}
\end{figure}
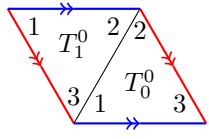

We also record the cusp triangulation in Figure \ref{cusp0_Berge} of the unknotted cusp of the Berge manifold. Call this cusp $c_0$. Generators for the first integral homology of this torus cusp cross-section are shown as the red and blue directed edges in Figure \ref{cusp0_Berge}. Note that the holonomy equation for this cusp is $-\zeta_1(z_1) z_0\zeta_2(z_0)=1$ which is equivalent to 
\begin{equation}
z_0=z_1.\label{Bcomhol}
\end{equation}
This means the subvariety $\mathcal{V}_0(M)$ of $\mathcal{V}(M)$ where cusp $c_0$ (corresponding to the unknotted component in Figure \ref{link_pics}) is complete, is the set of  $4$-tuples $(z_0,z_1,z_2,z_3)$ in  $\mathcal{V}(M)$ which further satisfy equation \ref{Bcomhol}.

\subsection{$(-2,3,8)$-pretzel link complement}\label{pretzel_gluing}

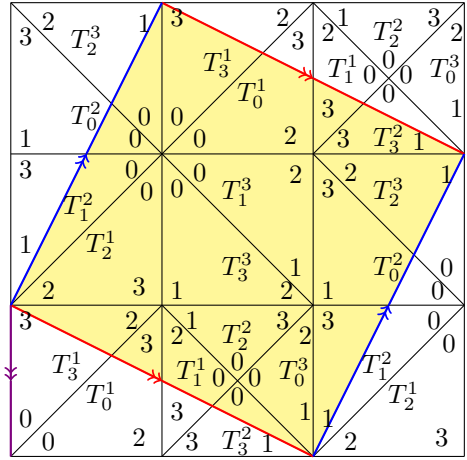
\begin{figure}
\begin{tikzpicture}
\filldraw[fill=yellow!50](0,2) --(2,6)--(6,4)--(4,0)--cycle;

\draw (0,0)--(6,0);
\draw (0,0)--(0,6);
\draw (6,0)--(6,6);
\draw (0,6)--(6,6);
\draw (2,0)--(2,6);
\draw (4,0)--(4,6);
\draw (0,2)--(6,2);
\draw (0,4)--(6,4);
\draw (0,2)--(4,6);
\draw (0,6)--(4,2);

\draw (0,0)--(2,2);
\draw (2,0)--(4,2);
\draw (4,4)--(6,2);
\draw (4,4)--(6,6);
\draw (2,2)--(4,0);
\draw (4,6)--(6,4);
\draw (4,0)--(6,2);

\draw[thick, ->>,color=red] (0,2)--(2,1);
\draw[thick, color=red] (2,1)--(4,0);
\draw[thick, ->>,color=red] (2,6)--(4,5);
\draw[thick, color=red] (4,5)--(6,4);

\draw[thick, ->>, color=blue] (0,2)--(1,4);
\draw[thick, color=blue] (1,4)--(2,6);
\draw[thick, ->>,color=blue] (4,0)--(5,2);
\draw[thick, color=blue] (5,2)--(6,4);

\draw[thick, violet, ->>] (0,2)--(0,1);
\draw [thick, violet](0,1)--(0,0);

\node at (.75, 1.2){$T_3^1$};
\node at (.2, .5){$0$};
\node at (.2, 1.8){$3$};
\node at (1.6, 1.8){$2$};

\node at (1.2, .8){$T_0^1$};
\node at (.5, .2){$0$};
\node at (1.8, 1.5){$3$};
\node at (1.7, .25){$2$};

\node at (.9, 3.3){$T_1^2$};
\node at (.2, 2.8){$1$};
\node at (.2, 3.8){$3$};
\node at (1.6, 3.8){$0$};

\node at (1.2, 2.8){$T_2^1$};
\node at (.5, 2.2){$2$};
\node at (1.8, 3.5){$0$};
\node at (1.7, 2.25){$3$};


%

\node at (1, 4.5){$T_0^2$};
\node at (.2,5.55){$3$};
\node at (1.65, 4.2){$0$};
\node at (.2, 4.2){$1$};

\node at (1, 5.5){$T_2^3$};
\node at (1.77,5.73){$1$};
\node at (.5,5.75){$2$};
\node at (1.77, 4.5){$0$};

\node at (3, 2.5){$T_3^3$};
\node at (2.2,3.55){$0$};
\node at (3.65, 2.2){$2$};
\node at (2.2, 2.2){$1$};

\node at (3, 3.5){$T_1^3$};
\node at (3.77,3.73){$2$};
\node at (2.5,3.75){$0$};
\node at (3.77, 2.5){$1$};

\node at (5, 2.5){$T_0^2$};
\node at (4.2,3.55){$3$};
\node at (5.65, 2.2){$0$};
\node at (4.2, 2.2){$1$};

\node at (5, 3.5){$T_2^3$};
\node at (5.77,3.73){$1$};
\node at (4.5,3.75){$2$};
\node at (5.77, 2.5){$0$};

\node at (4.85, 1.2){$T_1^2$};
\node at (4.2, .5){$1$};
\node at (4.2, 1.8){$3$};
\node at (5.6, 1.8){$0$};

\node at (5.2, .8){$T_2^1$};
\node at (4.5, .2){$2$};
\node at (5.8, 1.5){$0$};
\node at (5.7, .25){$3$};

\node at (2.75, 5.2){$T_3^1$};
\node at (2.2, 4.5){$0$};
\node at (2.2, 5.8){$3$};
\node at (3.6, 5.8){$2$};

\node at (3.2, 4.8){$T_0^1$};
\node at (2.5, 4.2){$0$};
\node at (3.8, 5.5){$3$};
\node at (3.7, 4.25){$2$};

\node at (2.4,1.12){$T_1^1$};
\node at (2.2,1.6){$2$};
\node at (2.2,.6){$3$};
\node at (2.75,1.05){$0$};

\node at (3.75,1.12){$T_0^3$};
\node at (3.9,1.6){$2$};
\node at (3.9,.6){$1$};
\node at (3.23,1.05){$0$};

\node at (3,1.6){$T_2^2$};
\node at (2.4,1.8){$1$};
\node at (3.6,1.8){$3$};
\node at (3,1.25){$0$};

\node at (3,.2){$T_3^2$};
\node at (2.4,.2){$3$};
\node at (3.4,.17){$1$};
\node at (3,.8){$0$};

\node at (4.4,5.12){$T_1^1$};
\node at (4.2,5.6){$2$};
\node at (4.2,4.6){$3$};
\node at (4.75,5.05){$0$};

\node at (5.75,5.12){$T_0^3$};
\node at (5.9,5.6){$2$};
\node at (5.9,4.6){$1$};
\node at (5.23,5.05){$0$};

\node at (5,5.6){$T_2^2$};
\node at (4.4,5.8){$1$};
\node at (5.6,5.8){$3$};
\node at (5,5.25){$0$};

\node at (5,4.2){$T_3^2$};
\node at (4.4,4.2){$3$};
\node at (5.4,4.17){$1$};
\node at (5,4.8){$0$};

\end{tikzpicture}
\caption{Cusp triangulation of the knotted cusp of the $(-2,3,8)$-pretzel link complement}
\label{cusp1_pretzel}
\end{figure}

We are going to use $N$ to denote the $(-2,3,8)$-pretzel link complement. A picture this link can be found in \cite[Row 2 in Table A.1]{MartPet}. The $(-2,3,8)$-pretzel link is also known as the link \texttt{L13n5885} in the Hoste-Thistlethwaite notation (see the right picture of Figure \ref{link_pics}). One could retrieve the SnapPy ideal triangulation of the this link complement by entering 
$$\texttt{Manifold(`L13n5885').\_get\_tetrahedra\_gluing\_data()}$$ in SnapPy \cite{snappy}. The triangulation isomorphism signature for this triangulation is 
$$\texttt{eLPkbcdddlfffg\_baBBdeaBBh}.$$
There are four tetrahedra in this triangualtion\textemdash appending \texttt{.tetrahedra\_shapes()} one can check from SnapPy \cite{snappy} that all four of them have shape parameters in $\{\mathrm{i}, 1+\mathrm{i}, \frac{1+\mathrm{i}}{2}\}$. Similarly as before, we are going to use $T_0, T_1, T_2$ and $T_3$ to denote the four ideal tetrahedra in this triangulation indexed as respectively $0$, $1$, $2$ and $3$ in SnapPy \cite{snappy}. SnapPy \cite{snappy} labels the (ideal) vertices of these (ideal) tetrahedra by $0$, $1$, $2$ and $3$. To write the edge equations in terms of the shape parameters, we will use the following choice: $z_0$ will denote shape parameter of the edge joining (ideal) vertices $0$ and $3$ of $T_0$, $z_1$ will denote the shape parameter of the edge joining (ideal) vertices $0$ and $1$ of $T_1$, and $z_2$ (respectively, $z_3$) will denote the shape parameter of the edge joining (ideal) vertices $0$ and $2$ of $T_2$ (respectively, $T_3$). These choices yield the following edge equations 
\begin{align}
z_0 z_1 z_2 z_3&=1\label{Pe2}\\
\zeta_1(z_0) \zeta_1(z_1)&=\zeta_1(z_2)\zeta_1(z_3).\label{Pe0}
\end{align}
In other words, the gluing variety $\mathcal{V}(N)$ of the $(-2,3,8)$-pretzel link complement consists of precisely the set of $4$-tuples $(z_0,z_1,z_2,z_3)$ in $\left(\mathbb{C}-\{0,1\}\right)^4$ cut out by equations \ref{Pe2} and \ref{Pe0}.

Cusp $0$ (respectively, cusp $1$) of this SnapPy triangulation corresponds to the unknotted component (respectively, knotted component) of the link. Figure \ref{cusp1_pretzel} shows the cusp triangulation of the cusp corresponding to the knotted complement obtained from this triangulation. We are going to denote this cusp by $c_1'$. In Figure \ref{cusp1_pretzel}, the generators of the first homology of the torus cusp cross-section are showed as the red and blue directed edges. Using these edges, we can read off the holonomy equation for this cusp from Figure \ref{cusp1_pretzel} as $z_3\zeta_2(z_2)\zeta_1(z_2)\zeta_2(z_3)\zeta_1(z_1)\zeta_2(z_0)\zeta_1(z_3)=1$, which is equivalent to 
\begin{equation}
z_2=\zeta_1(z_1) \zeta_2(z_0).\label{Pretzel_hol_eq}
\end{equation}

So, the subvariety $\mathcal{V}_0(N)$ of $\mathcal{V}(N)$ where cusp $c_1'$ (corresponding to the knotted component shown in the right of Figure \ref{link_pics}) is complete is the subset of $\mathcal{V}(N)$ consisting of $4$-tuples $(z_0,z_1,z_2,z_3)$ further satisfying equation \ref{Pretzel_hol_eq}.


\section{Parametrization of $\mathcal{V}_0(M)$ and Proof of Theorem \ref{Bcusptrace}}\label{Berge V0 parametrization}
We begin this section with the aim of parametrizing the one dimensional subvariety $\mathcal{V}_0(M)$ by a single parameter. In view of the technique used for $\Sp^3-6^2_2$ in \cite{CDHMMW}, we will use $\sigma_1$, $\sigma_2$, $\sigma_1'$ and $\sigma_2'$ to denote $z_0+z_1$, $z_0 z_1$, $z_2+z_3$ and $z_2 z_3$ respectively. In other words, $$(\sigma_1, \sigma_2, \sigma_1', \sigma_2')=\Phi(z_0,z_1,z_2,z_3)\coloneq \left(\phi(z_0,z_1), \phi(z_2,z_3)\right).$$ 

Now, Equation \ref{Be0} is equivalent to 
\begin{equation}
\sigma_2=\frac{\sigma_2'-\sigma_1'+1}{\sigma_2'}. \label{sigma_2}
\end{equation}
Similarly, Equation \ref{Be1} is equivalent to 
\begin{equation}
\sigma_2'=\frac{\sigma_2-\sigma_1+1}{\sigma_2}. \label{sigmaprime_2}
\end{equation}
Define $\mathcal{U}(M)$ to be the set of $4$-tuples $(\sigma_1, \sigma_2, \sigma_1', \sigma_2')$ in $(\mathbb{C}-\{0,1\})^4$ satisfying the above two equations. Note that the map $\Phi$ sending $(z_0,z_1, z_2,z_3)$ to $\Phi(z_0,z_1, z_2,z_3)$ is a regular map from $\mathcal{V}(M)$ to $\mathcal{U}(M)$. 
The above two equations together imply 
\begin{equation}
\sigma_2'-\sigma_1'=\sigma_2-\sigma_1. \label{sigmarelations}
\end{equation}

Equation \ref{Bcomhol} implies 
\begin{align}
\sigma_1=2 z_1\quad & \text{ and } \quad \sigma_2=\frac{\sigma_1^2}{4}=z_1^2. \label{sigmaz1}
\end{align}
The two equations above give us 
\begin{align*}
\sigma_2'=\frac{\sigma_1^2}{4}-\sigma_1+\sigma_1'.
\end{align*}

On the other hand, using Equation \ref{sigma_2} and Equation \ref{sigmaz1}, we get, 
\begin{equation}
\sigma_2'=\frac{4 (1-\sigma_1')}{(\sigma_1^2-4)}.\label{sigma2}
\end{equation}

Comparison of the above two equations together with the fact that $\sigma_1=2z_1$ result in 
\begin{equation}\label{sigma_1primeberge}
\sigma_1'=-z_1^2+2z_1+1- \frac{2}{z_1}+\frac{1}{z_1^2}.
\end{equation}
If we use the above formula in Equation \ref{sigma2}, we get, 
\begin{equation}\label{sigma_2primeberge}
\sigma_2'=\left(\frac{z_1-1}{z_1}\right)^2=1-\frac{2}{z_1}+\frac{1}{z_1^2}.
\end{equation}
We record the following lemma which we will need to use later. 
\begin{lemma}
$\Q (\sigma_1, \sigma_2)=\Q (\sigma_1', \sigma_2')=\Q(z_1).$
\end{lemma}
\begin{proof}
The first equality follows from equations \ref{sigma_2}, \ref{sigmaprime_2} and \ref{sigmarelations}. Since, $\sigma_1=2z_1$, 
the last equality then follows from the above two equations. 
\end{proof}

Now, since $\sigma_1'$ and $\sigma_2'$ are rational functions of $z_1$ given by equations \ref{sigma_1primeberge} and \ref{sigma_2primeberge} respectively, motivated by the work of \cite{CDHMMW} (paragraph after Lemma 6.5 therein), we will define the rational map $\Psi: \C \to \C^4$ as 
\begin{align}\label{BergesigmaLaurent}
\Psi(z_1)&=\left(2z_1, z_1^2, -z_1^2+2z_1+1- \frac{2}{z_1}+\frac{1}{z_1^2}, 1-\frac{2}{z_1}+\frac{1}{z_1^2}\right)=(\sigma_1, \sigma_2, \sigma_1', \sigma_2').
\end{align}
Observe that the image of $\Psi$ is a sub-variety of $\mathcal{U}(M)$ and equals the image of $\mathcal{V}_0(M)$ under the regular map $\Phi$. Denote this sub-variety as $\mathcal{U}_0(M)$. $\Psi$ induces the map $\Psi^{\ast}: \mathbb{C}(\mathcal{U}_0(M))\to \mathbb{C}(z_1)$ where $\mathbb{C}(\mathcal{U}_0(M))$ is the function field of $\mathcal{U}_0(M)$. 
\begin{remark}
We note that although a single shape parameter can be used here to parametrize the one-dimensional variety corresponding to the one-cusped Dehn fillings (as in the Whitehead case in \cite{NeumReid}), this parametrization is anchored here in a more nuanced way. 
\end{remark}

Before we proceed, we make a note of cusp parameter function on $\mathcal{V}_0(M)$ (see \cite[Proposition 1.6]{CDM}). Let $\mathbf{m}_0$ and $\mathbf{l}_0$ denote respectively the red and blue directed edges in Figure \ref{cusp0_Berge}. Then $(\mathbf{m}_0, \mathbf{l}_0)$ is an oriented generating pair for the first homology of the cusp torus corresponding to the unknotted component of the Berge manifold. We choose the red directed edge as the reference edge as in the proof of \cite[Proposition 1.6]{CDM} to conclude from this proposition that cusp parameter function defined on $\mathcal{V}_0(M)$ is $\tau=\zeta_2(z_1)$. This gives us the following fact. 
\begin{fact}\label{Bfilledshape}
If $(z_0,z_1,z_2,z_3) \in \mathcal{V}_0(M)$ corresponds to some hyperbolic $(p,q)$-Dehn filling on the cusp corresponding to the knotted component of the Berge manifold as in Figure \ref{link_pics}, then the cusp field of the corresponding filled manifold is $\Q(z_1)$. 
\end{fact}
We now turn to understand the trace fields of the hyperbolic one-cusped Dehn fillings of the Berge manifold. Using Poincar\'e polyhedron theorem, we will be able to identify the elements which generate these trace fields. Consequently, as we will see in Theorem \ref{Bcusptrace} below, these trace fields coincide with the corresponding cusp fields. Our proof method will follow that in \cite[Theorem 6.2]{NeumReid} and \cite[Proposition 6.16]{CDHMMW}.  
\begin{thm}\label{Bcusptrace}
\bergecusptrace
\end{thm}
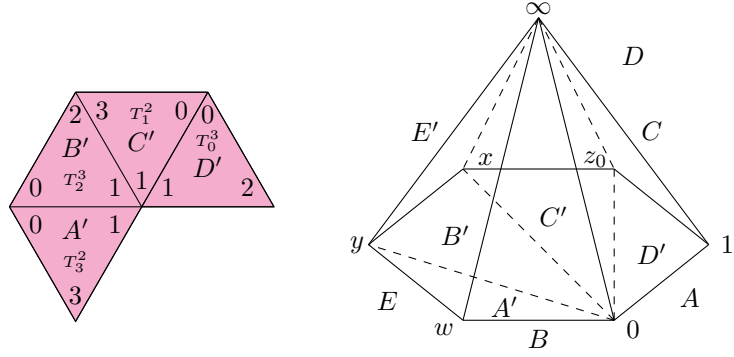
\begin{figure}

\begin{tikzpicture}
\begin{scope}[xshift=-4cm, yshift=1.5cm, scale=1.75]
\filldraw[fill=magenta!40](0,0) --(.5,.866)--(1.5, .866)--(2,0)--(1,0)--(.5,-.866)--cycle;

\draw (0,0)--(1,0);
\draw (0,0)--(.5,.866);
\draw (1,0)--(.5,.866);
\draw (1,0)--(1.5,.866);
\draw (.5,.866)--(1.5,.866);

\node at (.5, .45) {$B'$};
\node at (.5, .2) {\tiny{$T_2^3$}};
\node at (.2, .15){$0$};
\node at (.5, .7){$2$};
\node at (.8, .15){$1$};

\node at (1, .5) {$C'$};
\node at (1, .7) {\tiny{$T_1^2$}};
\node at (.7, .73){$3$};
\node at (1.3, .73){$0$};
\node at (1, .2){$1$};

\begin{scope}[xshift=1cm]
\draw (0,0)--(1,0);
\draw (0,0)--(.5,.866);
\draw (1,0)--(.5,.866);

\node at (.5, .3) {$D'$};
\node at (.5, .5) {\tiny{$T_0^3$}};
\node at (.2, .15){$1$};
\node at (.5, .7){$0$};
\node at (.8, .15){$2$};
\end{scope}

\begin{scope}[xshift=-.5cm,yshift=-.866cm]
\draw (1,0)--(.5,.866);
\draw (1,0)--(1.5,.866);
\node at (1, .7) {$A'$};

\node at (1, .45) {\tiny{$T_3^2$}};
\node at (.7, .73){$0$};
\node at (1.3, .73){$1$};
\node at (1, .2){$3$};
\end{scope}
\end{scope}

\begin{scope}[xshift=2cm, yscale=.5]
\draw (0,0)--(-1.25,2);
\draw (-1.25,2)--(0,4);
\draw (0,4)--(2,4);
\draw (2,4)--(3.25,2);
\draw (3.25,2)--(2,0);
\draw (2,0)--(0,0);
\draw (1,8)--(0,0);
\draw (1,8)--(-1.25,2);
\draw [dashed](1,8)--(0,4);
\draw [dashed] (1,8)--(2,4);
\draw (1,8)--(3.25,2);
\draw (1,8)--(2,0);
\draw [dashed](2,0)--(-1.25,2);
\draw [dashed](2,0)--(0,4);
\draw [dashed](2,0)--(2,4);

\node at (-1.4,2) {$y$};
\node at (-.25,-.25) {$w$};
\node at (2.25,-.25) {$0$};
\node at (3.5,2) {$1$};
\node at (.3,4.25) {$x$};
\node at (1.75,4.25) {$z_0$};
\node at (1,8.25) {$\infty$};

\node at (-.1,2.25) {$B'$};
\node at (1.2,2.75) {$C'$};
\node at (2.5,1.75) {$D'$};
\node at (.55,.35) {$A'$};

\node at (-1,.5) {$E$};
\node at (1,-.5) {$B$};
\node at (3,.6) {$A$};

\node at (-.5,5) {$E'$};
\node at (2.5,5) {$C$};
\node at (2.25,7) {$D$};

\end{scope}
\end{tikzpicture}

\caption{Left: Four cusp triangles from the knotted cusp triangulation each corresponding to an ideal tetrahedron in the tetrahedral decomposition of the Berge manifold, Right: 
Ideal decahedron representing the Berge manifold with sides faces $A$, $B$, $C$, $D$, $E$, $E'$ and bottom faces $A'$, $B'$, $C'$, $D'$}
\label{decaberge}
\end{figure}
\begin{proof}
Leveraging Fact \ref{sym_fact}, we will assume without loss of generality that $M_{(p,q)}$ is obtained by $(p,q)$-Dehn filling on the knotted cusp. 
Left of Figure \ref{decaberge} highlights one triangle from each tetrahedron in the cusp triangulation picture for the knotted cusp shown in Figure \ref{cusp1_Berge}. After realizing these highlighted triangles as the bottom faces of the corresponding tetrahedra and $\infty$ as their common vertex, 
we see that the Berge manifold can be obtained from the decahedron in right picture of Figure \ref{decaberge} by gluing pairs of faces $(A,A')$, $(B,B')$, $(C,C')$, $(D,D')$ and $(E,E')$.  The other vertices of this decahedron are at $0$, $1$, $z_0$, $x$, $y$ and $w$ where 
\begin{align*}
x=z_0 z_1, \quad y=z_0 z_1 z_2, \quad w=z_0 z_1 z_2 \zeta_1(z_3).
\end{align*}
Let $a,b,c,d,e$ denote the isometries sending respectively $A$ to $A'$, $B$ to $B'$, $C$ to $C'$, $D$ to $D'$ and $E$ to $E'$. Let $\Gamma_M$ denote the fundamental group of the Berge manifold. Then Poincar\'e polyhedron theorem implies that
\begin{equation*}
\Gamma_M=\left \langle a, e: a e^{-1}a^2 a^{-1}e^{-2}=a^{-1}e^{-2}ae^{-1}a^2\right \rangle.
\end{equation*}
(For the forthcoming computations in this proof, we used the help of Sage \cite{sagemath}.) 
Using the position of the vertices of the decahedron, we see that the fundamental group of $M_{(p,q)}$, to be denoted here as $\Gamma_{M_{(p,q)}}$, will have the presentation as above where the $\operatorname{SL}(2,\C)$-representation of $a$ and $e$ are
\begin{align*}
a&=\sqrt{\frac{w}{yw-y^2}}\begin{pmatrix} 0 & y\\ \frac{y-w}{w} &1\end{pmatrix}\\
e&=\frac{1}{\sqrt{(x-y)(w-y)}}\begin{pmatrix}x & -xw+yw-y^2 \\ 1 & -y\end{pmatrix}.
\end{align*} 

We can use \cite[Equation 3.25]{MacRe} to see that the trace field of $\Gamma_{M_{(p,q)}}$ is generated by $\operatorname{trace}(a)$, $\operatorname{trace}(e)$ and $\operatorname{trace}(ae)$. After simplification, we get $\operatorname{trace}(a)=\sqrt{\frac{w}{yw-y^2}}=\frac{1}{\sqrt{\sigma_2 \sigma_2'}}=\frac{1}{z_1-1}$, $\operatorname{trace}(e)=\sqrt{\frac{x-y}{w-y}}=z_1$ and $\operatorname{trace}(ae)=\frac{xw-yw+y^2}{\sqrt{yw(x-y)}}=(1-\sigma_2')\sqrt{\frac{\sigma_2}{1-\sigma_1'+\sigma_2'}}=\frac{2z_1-1}{z_1(z_1-1)}$. So, the trace field of $\Gamma_{M_{(p,q)}}$ is $\Q(z_1)$. But, the cusp field of  $\Gamma_{M_{(p,q)}}$ is $\Q(z_1)$ by Fact \ref{Bfilledshape}. Since the cusp field field is a subfield of the invariant trace field (\cite[Proposition 2.7]{NeumReid}), which itself is a subfield of the trace field, we conclude that cusp field, the invariant trace field and the trace field of $\Gamma_{M_{(p,q)}}$ are all equal to $\Q(z_1)$. 
\end{proof}
\begin{remark}
The above theorem implies that knot complements obtained from Dehn filling the unknotted component of the Berge manifold (as shown in the left picture of Figure \ref{link_pics}) have the same the trace field and cusp field. As mentioned in \cite{Hoffman_3comm}, the set of these knots includes various Berge knots. 
\end{remark}

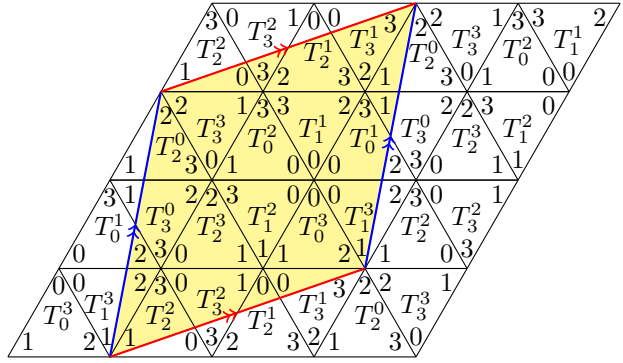
\begin{figure}
\begin{tikzpicture}[scale=1.35]

\filldraw[fill=yellow!50](1,0) --(1.5,2.598)--(4,3.464)--(3.5,.866)--cycle;

\draw (0,0)--(1,0);
\draw (0,0)--(.5,.866);
\draw (.5,.866)--(1,0);
\draw (1,0)--(1.5,.866);
\draw (.5,.866)--(1.5,.866);

\draw (1,0)--(2,0);
\draw (1.5,.866)--(2,0);
\draw (2,0)--(2.5, .866);
\draw (1.5,.866)--(2.5,.866);

\draw (2,0)--(3,0);
\draw (2.5,.866)--(3,0);
\draw (3,0)--(3.5, .866);
\draw (2.5,.866)--(3.5,.866);

\draw (3,0)--(4,0);
\draw (3.5,.866)--(4,0);
\draw (4,0)--(4.5, .866);
\draw (3.5,.866)--(4.5,.866);

\node at (.5, .4) {$T_0^3$};
\node at (.2, .15){$1$};
\node at (.5, .7){$0$};
\node at (.8, .15){$2$};

\node at (.9, .5) {$T_1^3$};
\node at (.7, .73){$0$};
\node at (1.3, .73){$2$};
\node at (.97, .2){$1$};

\node at (1.5, .4) {$T_2^2$};
\node at (1.2, .19){$1$};
\node at (1.5, .7){$3$};
\node at (1.8, .15){$0$};

\node at (2, .55) {$T_3^2$};
\node at (1.7, .73){$0$};
\node at (2.3, .73){$1$};
\node at (2, .2){$3$};

\node at (2.5, .35) {$T_2^1$};
\node at (2.2, .15){$2$};
\node at (2.5, .7){$0$};
\node at (2.8, .15){$3$};

\node at (3, .5) {$T_3^1$};
\node at (2.7, .73){$0$};
\node at (3.25, .67){$3$};
\node at (3, .2){$2$};

\node at (3.55, .4) {$T_2^0$};
\node at (3.2, .15){$1$};
\node at (3.5, .7){$2$};
\node at (3.8, .15){$3$};

\node at (4, .5) {$T_3^3$};
\node at (3.7, .73){$2$};
\node at (4.3, .73){$1$};
\node at (4, .2){$0$};

\begin{scope}[xshift=.5 cm, yshift=.866 cm]
\draw (0,0)--(1,0);
\draw (0,0)--(.5,.866);
\draw (.5,.866)--(1,0);
\draw (1,0)--(1.5,.866);
\draw (.5,.866)--(1.5,.866);

\draw (1,0)--(2,0);
\draw (1.5,.866)--(2,0);
\draw (2,0)--(2.5, .866);
\draw (1.5,.866)--(2.5,.866);

\draw (2,0)--(3,0);
\draw (2.5,.866)--(3,0);
\draw (3,0)--(3.5, .866);
\draw (2.5,.866)--(3.5,.866);

\draw (3,0)--(4,0);
\draw (3.5,.866)--(4,0);
\draw (4,0)--(4.5, .866);
\draw (3.5,.866)--(4.5,.866);

\node at (.5, .4) {$T_0^1$};
\node at (.2, .15){$0$};
\node at (.5, .7){$3$};
\node at (.8, .15){$2$};

\node at (1, .5) {$T_3^0$};
\node at (.7, .73){$1$};
\node at (1.3, .73){$2$};
\node at (1, .2){$3$};

\node at (1.5, .4) {$T_2^3$};
\node at (1.2, .15){$0$};
\node at (1.5, .7){$2$};
\node at (1.8, .15){$1$};

\node at (2, .5) {$T_1^2$};
\node at (1.7, .73){$3$};
\node at (2.3, .73){$0$};
\node at (2, .2){$1$};

\node at (2.5, .4) {$T_0^3$};
\node at (2.2, .15){$1$};
\node at (2.5, .7){$0$};
\node at (2.8, .15){$2$};

\node at (2.95, .5) {$T_1^3$};
\node at (2.7, .73){$0$};
\node at (3.3, .73){$2$};
\node at (2.97, .21){$1$};

\node at (3.5, .4) {$T_2^2$};
\node at (3.2, .15){$1$};
\node at (3.5, .7){$3$};
\node at (3.8, .15){$0$};

\node at (4, .5) {$T_3^2$};
\node at (3.7, .73){$0$};
\node at (4.3, .73){$1$};
\node at (4, .2){$3$};

\end{scope}

\begin{scope}[xshift=1 cm, yshift=1.732 cm]
\draw (0,0)--(1,0);
\draw (0,0)--(.5,.866);
\draw (.5,.866)--(1,0);
\draw (1,0)--(1.5,.866);
\draw (.5,.866)--(1.5,.866);

\draw (1,0)--(2,0);
\draw (1.5,.866)--(2,0);
\draw (2,0)--(2.5, .866);
\draw (1.5,.866)--(2.5,.866);

\draw (2,0)--(3,0);
\draw (2.5,.866)--(3,0);
\draw (3,0)--(3.5, .866);
\draw (2.5,.866)--(3.5,.866);

\draw (3,0)--(4,0);
\draw (3.5,.866)--(4,0);
\draw (4,0)--(4.5, .866);
\draw (3.5,.866)--(4.5,.866);

\node at (.6, .32) {$T_2^0$};
\node at (.2, .15){$1$};
\node at (.55, .63){$2$};
\node at (.8, .15){$3$};

\node at (1, .5) {$T_3^3$};
\node at (.7, .73){$2$};
\node at (1.3, .73){$1$};
\node at (1, .2){$0$};

\node at (1.5, .4) {$T_0^2$};
\node at (1.2, .15){$1$};
\node at (1.5, .7){$3$};
\node at (1.8, .15){$0$};

\node at (2, .5) {$T_1^1$};
\node at (1.7, .73){$3$};
\node at (2.3, .73){$2$};
\node at (2, .2){$0$};

\node at (2.5, .4) {$T_0^1$};
\node at (2.2, .15){$0$};
\node at (2.5, .7){$3$};
\node at (2.8, .15){$2$};

\node at (3, .5) {$T_3^0$};
\node at (2.7, .73){$1$};
\node at (3.3, .73){$2$};
\node at (3, .2){$3$};

\node at (3.5, .4) {$T_2^3$};
\node at (3.2, .15){$0$};
\node at (3.5, .7){$2$};
\node at (3.8, .15){$1$};

\node at (4, .5) {$T_1^2$};
\node at (3.7, .73){$3$};
\node at (4.3, .73){$0$};
\node at (4, .2){$1$};
\end{scope}

\begin{scope}[xshift=1.5 cm, yshift=2.598 cm]
\draw (0,0)--(1,0);
\draw (0,0)--(.5,.866);
\draw (.5,.866)--(1,0);
\draw (1,0)--(1.5,.866);
\draw (.5,.866)--(1.5,.866);

\draw (1,0)--(2,0);
\draw (1.5,.866)--(2,0);
\draw (2,0)--(2.5, .866);
\draw (1.5,.866)--(2.5,.866);

\draw (2,0)--(3,0);
\draw (2.5,.866)--(3,0);
\draw (3,0)--(3.5, .866);
\draw (2.5,.866)--(3.5,.866);

\draw (3,0)--(4,0);
\draw (3.5,.866)--(4,0);
\draw (4,0)--(4.5, .866);
\draw (3.5,.866)--(4.5,.866);

\node at (.5, .4) {$T_2^2$};
\node at (.24, .2){$1$};
\node at (.5, .7){$3$};
\node at (.8, .15){$0$};

\node at (1, .58) {$T_3^2$};
\node at (.7, .73){$0$};
\node at (1.3, .73){$1$};
\node at (1, .2){$3$};

\node at (1.55, .37) {$T_2^1$};
\node at (1.2, .15){$2$};
\node at (1.5, .7){$0$};
\node at (1.8, .15){$3$};

\node at (2, .5) {$T_3^1$};
\node at (1.7, .73){$0$};
\node at (2.25, .67){$3$};
\node at (2, .2){$2$};

\node at (2.58, .36) {$T_2^0$};
\node at (2.2, .15){$1$};
\node at (2.55, .63){$2$};
\node at (2.8, .15){$3$};

\node at (3, .5) {$T_3^3$};
\node at (2.7, .73){$2$};
\node at (3.3, .73){$1$};
\node at (3, .2){$0$};

\node at (3.5, .4) {$T_0^2$};
\node at (3.2, .15){$1$};
\node at (3.5, .7){$3$};
\node at (3.8, .15){$0$};

\node at (4, .5) {$T_1^1$};
\node at (3.7, .73){$3$};
\node at (4.3, .73){$2$};
\node at (4, .2){$0$};
\end{scope}

\draw [blue, thick, ->>](1,0)--(1.25, 1.299);
\draw [blue, thick](1.25,1.299)--(1.5, 2.598);

\draw [blue, thick, ->>](3.5,.866)--(3.75, 2.165);
\draw [blue, thick](3.75,2.165)--(4, 3.464);

\draw [red, thick, ->>](1,0)--(2.25, .433);
\draw [red, thick](2.25,.433)--(3.5,.866);

\draw [red, thick, ->>](1.5,2.598)--(2.75, 3.031);
\draw [red, thick](2.75,3.031)--(4, 3.464);

\end{tikzpicture}
\caption{Cusp triangulation of the knotted cusp of the Berge manifold}
\label{cusp1_Berge}
\end{figure}

\section{Arithmetic Data of the one-cusp Dehn fillings of the Berge manifold}\label{arithmetic berge}
One can identify the Berge manifold (i.e. the link complement in the left picture of Figure \ref{link_pics}) on SnapPy \cite{snappy} with \texttt{m202}. Let $\mathbf{m}_0^{202}$ and $\mathbf{l}_0^{202}$ be the meridian and the longitude chosen by SnapPy \cite{snappy} for the triangulation \texttt{Manifold(`m202')} for its cusp $0$. The isomorphism signature of this triangulation is 
$$\texttt{`eLMkbbdddemdxi\_BaaBbBba'}. $$
One can check by entering the command $$\texttt{Manifold(`m202').symmetry\_group().isometries()}$$ on SnapPy \cite{snappy} that there is an orientation preserving self-isometry $g$ of the Berge manifold\footnote{Unfortunately, as pointed out by SnapPy \cite{snappy}, this isometry does not extend to the link.} which sends cusp $0$ corresponding to  \texttt{Manifold(`m202')} to itself such that $g(\mathbf{m}_0^{202})=\mathbf{m}_0^{202} \left (\mathbf{l}_0^{202}\right)^{-1}$ and $g(\mathbf{l}_0^{202})=\mathbf{m}_0^{202}$ . One can further check from SnapPy \cite{snappy} that $g$ has order $6$. Now, by entering the command $$\texttt{Manifold(`m202').cusp\_neighborhood().view()}$$ on SnapPy \cite{snappy}, one can check that $\mathbf{m}_0^{202}$ and $\mathbf{l}_0^{202}$ form a hexagonal generating set for the peripheral subgroup $\Lambda_0^{202}$ of cusp $0$ of \texttt{Manifold(`m202')} with an acute angle between them. There are six possible hexagonal acute angled oriented pairs of generators of $\Lambda_0^{202}$. If $S_{202}$ denotes the set consisting of these six pairs, then, 
$$S_{202}=\left \{\left(g^j(\mathbf{m}_0^{202}), g^j(\mathbf{l}_0^{202})\right): j \in \{0, 1, \dots, 5\}\right \}$$ 
and the group generated by $g$ acts transitively on $S_{202}$. 
So $(p,q)$-Dehn filling with respect to a member of $S_{202}$ is isometric to the $(p,q)$-Dehn filling with respect to another member of $S_{202}$. This would also imply the following fact. 
\begin{fact}\label{Berge_nonnegativeDehn}
A $(p,q)$-Dehn filling with respect to a pair in $S_{202}$ is isometric to some $(p',q')$-Dehn filling with respect to the same pair where both $p'$ and $q'$ are non-negative. 
\end{fact}

We now look at the cusp triangulation of the knotted cusp of the Berge manifold as shown in the left of Figure \ref{link_pics}. This is cusp $1$ of the triangulation $$\texttt{`eLMkbbdddemdxi\_baBddIaBBrt'}.$$ Call this cusp $c_1$. The cusp triangulation of $c_1$ is given in Figure \ref{cusp1_Berge}. Let us denote the blue directed edge and the red directed edge in Figure \ref{cusp1_Berge} by $\mathbf{l}_1$ and $\mathbf{m}_1$ respectively. 
Note that  $(\mathbf{m}_1,\mathbf{l}_1)$ is a hexagonal acute angled oriented pairs of generators of the corresponding cusp group. One can check from SnapPy \cite{snappy} that there is an isomorphism from the triangulation \texttt{Manifold(`m202')} to the triangulation 
\texttt{`eLMkbbdddemdxi\_baBddIaBBrt'} (corresponding to the link drawn in SnapPy as shown in left of Figure \ref{link_pics}) sending cusp $0$ to cusp $1$ (and cusp $1$ to cusp $0$). So, $(\mathbf{m}_1, \mathbf{l}_1)$ belongs to $S_{202}$. Therefore, Fact \ref{sym_fact} and \ref{Berge_nonnegativeDehn} imply that we only need to study the $(p,q)$-Dehn fillings on the knotted cusp, where both $p$ and $q$ are non-negative, with respect to $(\mathbf{m}_1, \mathbf{l}_1)$ to prove Theorem \ref{berge_arith_thm}. 


The holonomy derivatives of $\mathbf{m}_1$ and $\mathbf{l}_1$ can be computed from Figure \ref{cusp1_Berge} as follows 
\begin{align*}
\mu(\mathbf{m}_1)=\frac{\zeta_1(z_2)}{\zeta_1(z_3)}, \qquad \mu(\mathbf{l}_1)=\frac{z_3}{z_2}.
\end{align*}
%
This means that the $(p,q)$-Dehn filling on the $c_1$ cusp with respect to the generating pair $(\mathbf{m}_1, \mathbf{l}_1)$ renders the Dehn filling equation 
\begin{equation}
\left(\frac{1-z_3}{1-z_2}\right)^p \left(\frac{z_3}{z_2}\right)^q=1.\label{Bdehneq} 
\end{equation}

We should note here that at the complete structure, the holonomy equation for cusp $c_1$ is $\mu(\mathbf{l}_1)=\frac{z_3}{z_2}=1$ as well. This equation along with equations \ref{Be0}, \ref{Be1} and \ref{Bcomhol} give us the following fact, which can be checked from SnapPy \cite{snappy} as well (as mentioned in the beginning of Subsection \ref{berge_gluing}). 
\begin{fact}\label{Berge_full_complete}
At the complete structure, i.e. when $M_{(p,q)}=M_{\infty}=M$, we have, $z_j=\frac{1+\sqrt{3}\mathrm{i}}{2}$ for each $j=0,1,2,3$. 
\end{fact}

Using the approach in the proof of \cite[Theorem 6.3]{NeumReid} and \cite[Proposition 6.6] {CDHMMW}, we will now show the existence a unit non-rational algebraic integer in the cusp field of the one-cusped hyperbolic Dehn fillings $M_{(p,q)}$. In particular, our goal here will be to find a polynomial in $z_1$ with both the leading coefficient and the constant term equal to $\pm 1$ such that the corresponding polynomial equation is equivalent to the $(p,q)$-Dehn filling equation above. We will first need to convert the Dehn filling equation in terms of $\sigma_1'$ and $\sigma_2'$ following the technique introduced in \cite{CDHMMW}. Since $\sigma_1'$ and $\sigma_2'$ can both be written as Laurent series in $z_1$ as in Equation \ref{BergesigmaLaurent}, the proposition below will follow. As in \cite{CDHMMW}, we will use the notations $g_{(p,q)}$, $h_{(p,q)}$ and $h'_{(p,q)}$ below. 


\begin{prop}\label{Balgint}

Let $(p,q)$ be a pair of non-negative integers. Then any $z_1$ corresponding to a $(p,q)$ hyperbolic Dehn filling on cusp $c_1$ of the Berge manifold along  $(\mathbf{m}_1, \mathbf{l}_1)$ is a unit non-rational algebraic integer. 
\end{prop}
\begin{proof}
Recall functions $s_d$ and $t_d$ of \cite{CDHMMW} from Subsection \ref{onecusp622}. Below, we will use them to denote respectively $s_d(z_2,z_3)$ and $t_d(\sigma_1',\sigma_2')$. We take $g_{(p,q)}$ and $h_{(p,q)}$ as below
\begin{align*}
g_{(p,q)}=\sum_{j=0}^{p} \begin{pmatrix}p\\ j \end{pmatrix} (-1)^j s_{q+j-1}\quad, \quad h_{(p,q)}=\sum_{j=0}^{p} \begin{pmatrix}p\\ j \end{pmatrix} (-1)^j t_{q+j-1}. 
\end{align*}

Equation \ref{Bdehneq} is equivalent to the equation $z_2^q (1-z_2)^p- z_3^q (1-z_3)^p=0$, which is equivalent to $(z_2-z_3) g_{(p,q)}=0$. Since $z_2=z_3$ corresponds to the complete structure at cusp $c_1$, we can take $g_{(p,q)}=0$ as the $(p,q)$-Dehn filling equation. One should note that $\phi^{\ast} (h_{(p,q)})=g_{(p,q)}$. Recalling $\Psi$ from Equation \ref{BergesigmaLaurent}, we obtain, $\Psi^{\ast} (h_{(p,q)})=z_1^{-2(p+q-1)}h'_{(p,q)}$ where $h'_{(p,q)}$ is a polynomial in $\Z[z_1]$ with leading term $\pm z_1^{4(p+q-1)}$ and constant term $\pm 1$. So, $z_1$ is a non-rational unit algebraic integer. 


\end{proof}

We will now use Proposition \ref{Balgint} to prove Theorem \ref{berge_arith_thm} where we show that no $M_{(p,q)}$ is arithmetic. We follow the approach taken in \cite[Corollary 6.14]{CDHMMW} for this proof. More specifically, we show that Proposition \ref{Balgint} obstructs the cusp field $\mathbb{Q}(z_1)$ from becoming quadratic imaginary. This will allow us to leverage Neumann and Reid's criterion \cite[Proposition 4.4(a)]{NeumReid} for arithmeticity in probing which $M_{(p,q)}$'s are arithmetic, as we will see below.  
\begin{thm}\label{berge_arith_thm}
\bergearith
\end{thm}

\begin{proof}

Let $O$ be a hyperbolic orbifold obtained by Dehn filling a single cusp of the Berge manifold. Using Fact \ref{sym_fact} and \ref{Berge_nonnegativeDehn}, we can assume that $O$ is obtained by $(p,q)$-Dehn filling on cusp $c_1$ along $(\mathbf{m}_1, \mathbf{l}_1)$ where both $p$ and $q$ are non-negative integers. Since $\tau=\zeta_2(z_1)$ is a cusp moduli of the corresponding hyperbolic Dehn filling (see the paragraph before Fact \ref{Bfilledshape}), $z_1$ lies in the upper-half plane as well. Suppose the cusp field $\Q(z_1)$ is quadratic imaginary. Then since $z_1$ is unit non-rational algebraic integer by Proposition \ref{Balgint}, $z_1$ is either $\mathrm{i}$ or $\frac{1+\sqrt{3}\mathrm{i}}{2}$ or $\frac{-1+\sqrt{3}\mathrm{i}}{2}$ (see Fact \ref{quad_unit}). Now, $z_1=\frac{1+\sqrt{3}\mathrm{i}}{2}$ can be ruled out as it corresponds to the complete structure (see Fact \ref{Berge_full_complete}). One can check that if $z_1$ is $\mathrm{i}$, then $\sigma_1'=4\mathrm{i}+1$, $\sigma_2'=2\mathrm{i}$, and consequently, 
$$\left \{z_2, z_3 \right \} =\left\{\frac{1+(4+\sqrt{15})\mathrm{i}}{2},  \frac{1+(4-\sqrt{15})\mathrm{i}}{2}\right \}.$$
Now one can verify that these values of $z_2$ and $z_3$ satisfy Equation \ref{Bdehneq} for non-negative integers $p$ and $q$ only if $p=q=0$. So, $z_1$ can't be $\mathrm{i}$. 

If $z_1$ is $\frac{-1+\sqrt{3}\mathrm{i}}{2}$, then $\sigma_1'=3 \sqrt{3}\mathrm{i}+1$ and $\sigma_2'=\frac{3+3\sqrt{3}\mathrm{i}}{2}$, and therefore, 
$$\left \{z_2, z_3 \right \} =\left\{\frac{1+(3\sqrt{3}+4\sqrt{2})\mathrm{i}}{2}, \frac{1+(3\sqrt{3}-4\sqrt{2})\mathrm{i}}{2}\right \}.$$ 
Similarly, as in the case above, one can check that for non-negative integers $p$ and $q$, Equation \ref{Bdehneq} is not satisfied for these values of $z_2$ and $z_3$ unless $p=q=0$. So, $z_1\ne \frac{-1+\sqrt{3}\mathrm{i}}{2}$. This proves the theorem. 

\end{proof}

\section{Parametrization of $\mathcal{V}_0(N)$ and Proof of Theorem \ref{tracecusp_pret}}\label{pretzel V0 parametrization}
We will now focus on obtaining a one variable parametrization of $\mathcal{V}_0(N)$, which has dimension one. Note that equations \ref{Pe2} and \ref{Pretzel_hol_eq} give us 
\begin{equation}
z_3=\zeta_2(z_1) \zeta_1(z_0). \label{z3eq_pret}
\end{equation}
As in the Berge manifold case, motivated by the work in \cite{CDHMMW}, we will use $\sigma_1$, $\sigma_2$, $\sigma_1'$  and $\sigma_2'$ to denote $z_0+z_1$, $z_0 z_1$, $z_2+z_3$ and $z_2 z_3$ respectively. Using equations \ref{Pretzel_hol_eq} and \ref{z3eq_pret}, we see that 
\begin{align}\label{sigma1peq1}
\sigma_1'=\frac{z_0-1}{z_0(1-z_1)}+\frac{z_1-1}{z_1(1-z_0)}=-\frac{z_1(1-z_0)^2+z_0(1-z_1)^2}{z_0 z_1(1-z_0)(1-z_1)}=\frac{-\sigma_1+4\sigma_2-\sigma_2 \sigma_1}{\sigma_2(1-\sigma_1+\sigma_2)}.
\end{align}

Equation \ref{Pe2} implies that $\sigma_2'=\sigma_2^{-1}$. On the other hand, Equation \ref{Pe0} gives us $\sigma_2 -\sigma_1=\sigma_2'-\sigma_1'$. This means
\begin{align}\label{sigma1peq2}
\sigma_1'=\sigma_1-(\sigma_2 -\sigma_2')=\sigma_1-(\sigma_2-\sigma_2^{-1})=\frac{\sigma_1 \sigma_2 -\sigma_2^2+1}{\sigma_2}
\end{align}

Equating $\sigma_1'$ from equations \ref{sigma1peq1} and \ref{sigma1peq2} gives us
\begin{align}
\sigma_1 \sigma_2 -\sigma_2^2+1&=\frac{-\sigma_1 +4 \sigma_2-\sigma_1 \sigma_2}{1-\sigma_1+\sigma_2}\nonumber\\
\implies \sigma_2(\sigma_1-\sigma_2)+1&=\frac{-1+3\sigma_2-\sigma_1 \sigma_2}{1-\sigma_1+\sigma_2}+1\nonumber\\
\implies \sigma_2(\sigma_1-\sigma_2)&=\frac{-1+3\sigma_2-\sigma_1 \sigma_2}{1-\sigma_1+\sigma_2}. \label{sigmarel_pret}
\end{align}

Using equations \ref{Pe0}, \ref{Pretzel_hol_eq} and \ref{z3eq_pret}, we also see that 
\begin{align}
 \frac{1-z_0 z_1}{z_0(1-z_1)} \frac{1-z_0 z_1}{z_1(1-z_0)}&=(1-z_0)(1-z_1)\nonumber\\
\implies (1-\sigma_2)^2&=\sigma_2(1-\sigma_1+\sigma_2)^2. \label{sigmarel_2_pret}
\end{align}

One can check that equations \ref{sigmarel_pret} and \ref{sigmarel_2_pret} are equivalent. Define $\mathcal{U}_0(N)$ to be the set of $2$-tuples $(\sigma_1, \sigma_2)$ in $(\mathbb{C}-\{0,1\})^2$ which satisfy Equation \ref{sigmarel_2_pret}. Observe that the map $\phi$ sending $(z_0,z_1)$ to $(z_0+z_1,z_0z_1)$ is a regular map from $\mathcal{V}_0(N)$ to $\mathcal{U}_0(N)$. 
%
Let's turn our attention to computing a cusp parameter function on $\mathcal{V}(N)$ following \cite[Proposition 1.6]{CDM}. We choose the red directed edge and blue directed edge in Figure \ref{cusp1_pretzel} to denote $\mathbf{m}_1$ and $\mathbf{l}_1$ respectively. Note that $(\mathbf{m}_1, \mathbf{l}_1)$ is an oriented generating pair for the first homology of the cusp torus corresponding to the knotted cusp of $\mathbb{S}^3-\texttt{L13n5885}$. We now take the violet directed edge in Figure \ref{cusp1_pretzel} as the reference edge and apply \cite[Proposition 1.6]{CDM} to see that 
\begin{align*}
\tau(\mathbf{m}_1)&=z_3-z_3\zeta_2(z_3) \zeta_1(z_0) \zeta_1(z_1)\zeta_2(z_2)+z_3\zeta_2(z_3) \zeta_1(z_0) \zeta_1(z_1)\zeta_2(z_2)\zeta_1(z_2) \zeta_2(z_0),\\
\tau(\mathbf{l}_1)&=z_3\zeta_2(z_2)-z_3\zeta_2(z_2) \zeta_1(z_2) \zeta_2(z_3) \zeta_1(z_1)\zeta_2(z_0) \zeta_1(z_3).
\end{align*}

If we denote the cusp moduli corresponding to the oriented pair $(\mathbf{m}_1, \mathbf{l}_1)$ for the knotted cusp of $\mathbb{S}^3-\texttt{L13n5885}$ by $\tau_{\mathcal{V}}$, then the above equations along with equations \ref{Pretzel_hol_eq} and \ref{z3eq_pret} imply that 
\begin{align*}
\tau_{\mathcal{V}}=\frac{\tau(\mathbf{l}_1)}{\tau(\mathbf{m}_1)}&=\frac{(z_0^2 z_1+z_0 z_1^2-3z_0 z_1 +1)(z_0-1)(z_1-1)}{z_0^2 z_1^2 -3z_0^2z_1-3z_0z_1^2+z_0^2+6z_0 z_1+z_1^2-z_0-z_1-1}.\nonumber\\
\end{align*}
\begin{lemma}\label{pret_birational}
There is a birational map $\tau: \mathcal{U}_0(N) \to \mathbb{C}$ such that $\phi^{\ast}(\tau)=\tau_{\mathcal{V}}$. 
\end{lemma}
\begin{proof}
We first define the rational map $\tau: \mathcal{U}_0(N) \to \mathbb{C}$ as follows
$$\tau(\sigma_1, \sigma_2)=\frac{1-\sigma_2}{\sigma_2-\sigma_1+1}.$$
Consider the rational map from $\mathbb{C}$ to $\mathcal{U}_0(N)$ that sends $\tau$ to $(\tau^2+\tau+1-\tau^{-1}, \tau^2)$. 
One can now check that this rational map and $\tau$ are rational inverses of each other. One direction is straightforward and for the other we need to use Equation \ref{sigmarel_2_pret}\textemdash for the first component we need to apply $(2-\sigma_1)\left(\sigma_2(1-\sigma_1+\sigma_2)^2-(1-\sigma_2)^2\right)=0$ in the numerator and for the second component we need to observe that $\sigma_2=\frac{(1-\sigma_2)^2}{(1-\sigma_1+\sigma_2)^2}$. This makes $\tau$ a birational map. 
Now, by using Equation \ref{sigmarel_2_pret}, multiplying both the numerator and denominator by $(\sigma_2-\sigma_1)(\sigma_2-\sigma_1+1)$ and applying Equation \ref{sigmarel_pret}, we see that 
\begin{align*}
\tau(\sigma_1, \sigma_2)=\frac{\sigma_2(\sigma_2-\sigma_1+1)}{1-\sigma_2}&=\frac{\sigma_2(\sigma_2-\sigma_1)(\sigma_2-\sigma_1+1)^2}{\left((\sigma_2-\sigma_1)-\sigma_2(\sigma_2-\sigma_1)\right)(\sigma_2-\sigma_1+1)}\\
&=\frac{(\sigma_2 \sigma_1-3\sigma_2+1)(\sigma_2-\sigma_1+1)}{(\sigma_2-\sigma_1+1)^2-(\sigma_2-\sigma_1+1)+(3\sigma_2-1-\sigma_1 \sigma_2)}.
\end{align*}
Simplifying the above equation, we get, 
$$\tau(\sigma_1, \sigma_2)=\frac{(\sigma_2 \sigma_1-3 \sigma_2+1)(\sigma_2-\sigma_1+1)}{\sigma_2^2-3\sigma_2 \sigma_1+\sigma_1^2+4\sigma_2-\sigma_1-1}.$$
It is now straightforward to see that $\phi^{\ast}(\tau)=\tau_{\mathcal{V}}$. 
\end{proof}

This gives us the following fact. 
\begin{fact}\label{pretzel_cusp_fld}
If $(z_0,z_1,z_2,z_3) \in \mathcal{V}_0(N)$ corresponds to a hyperbolic orbifold obtained by some $(p,q)$-Dehn filling on the unknotted cusp of the $(-2,3,8)$-pretzel link complement, then the cusp field of that hyperbolic orbifold is $\mathbb{Q}(\tau)$ where $\tau=\frac{1-z_0 z_1}{(1-z_0)(1-z_1)}$. 
\end{fact}

The proof of Lemma \ref{pret_birational} above also shows that 
\begin{equation}\label{tau_sigma1_pret}
\sigma_1=\tau^2+\tau+1-\tau^{-1} \quad \text{and } \sigma_2=\tau^2.
\end{equation}
Since $\sigma_2'=\sigma_2^{-1}$ and $\sigma_1-\sigma_2=\sigma_1'-\sigma_2'$, we have $\Q(\sigma_1, \sigma_2)=\Q(\sigma_1', \sigma_2')$. As a consequence, we can also conclude the following lemma.

\begin{lemma}\label{tau_sigma_pretzel}
$\Q(\sigma_1', \sigma_2')=\Q(\sigma_1, \sigma_2)=\Q(\tau)$.
\end{lemma}

As in the Berge manifold case, we will follow the approach taken in \cite{CDHMMW}. Recall Equation \ref{tau_sigma1_pret} from above. Toeing the discussion in \cite[Paragraph after Lemma 6.5]{CDHMMW}, we now define the rational map $\psi: \mathbb{C}\to \mathbb{C}^2$ as 
$$\psi (\tau)=(\tau^2+\tau+1-\tau^{-1}, \tau^2)=(\sigma_1, \sigma_2).$$
Observe that the image of $\psi$ lies in $\mathcal{U}_0(N)$. Furthermore, $\psi$ induces the pullback map $\psi^{\ast}: \mathbb{C}(\mathcal{U}_0(N)) \to \C(\tau)$. 
\begin{remark}
In this case, neither of any of the shape parameter (as in the Whitehead case in \cite{NeumReid}) nor either of the $\sigma_i$ (as in the $6^2_2$ case in \cite{CDHMMW}) parametrize the one-dimensional variety corresponding to the one-cusped Dehn fillings. 
\end{remark}

We will now focus on the trace fields of the one-cusped Dehn fillings of the $(-2,3,8)$-pretzel link complement. As it turns out in Theorem \ref{tracecusp_pret} below, these trace fields are same as the corresponding cusp fields. This is similar to what we saw in the Berge manifold case. Our proof will be similar\textemdash we will follow the approach in the proof of \cite[Theorem 6.2]{NeumReid} and \cite[Proposition 6.16]{CDHMMW} and use Poincar\'e polyhedron theorem. 

\begin{thm}\label{tracecusp_pret}
\pretzelcusptrace
\end{thm}

\begin{proof}
Fact \ref{sym_fact} implies that it is enough to prove the result when $N_{(p,q)}$ is obtained by $(p,q)$-Dehn filling cusp $c_0'$ of the $(-2,3,8)$-pretzel link complement (i.e. the unknotted cusp). Following the proof method of \cite[Proposition 6.16]{CDHMMW}, we will arrange the four ideal tetrahedra to form an octahedron whose face pairings result in $N$. For this arrangement, we use the four cusp triangles from the cusp triangulation of the unknotted cusp of the link complement, shown in Figure \ref{cusp0_pret}. The two pictures on Figure \ref{octpret} show how this octahedron is constructed. After choosing one of the vertex at $\infty$ and two other vertices at $0$ and $1$, we see from the left of Figure \ref{cusp0_pret} that the rest of the vertices of this octahedron should be at $z_1$, $z_1 z_2$ and $z_0 z_1 z_2$. 

\begin{figure}
\begin{tikzpicture}


\draw[green, very thick] (0,0)--(2,0);
\draw[brown, thick] (0,0)--(0,2);
\draw [orange, thick](0,2)--(2,2);
\draw[violet, thick] (2,2)--(2,0);
\draw (0,2)--(1,3.732);
\draw (2,2)--(1,3.732);
\draw (0,0)--(1,3.732);
\draw (2,0)--(1,3.732);

\draw[magenta, thick, dashed] (0,2)--(1,-1.732);
\draw [magenta, thick, dashed](2,2)--(1,-1.732);
\draw [magenta, thick](0,0)--(1,-1.732);
\draw [magenta, thick](2,0)--(1,-1.732);

\node at (-.1, -.15){$1$};

\node at (1.1, 3.85){$\infty$};
\node at  (1.1,-1.85){$0$};
\node at  (2.15,-.15){$z_1$};
\node at  (2.3,2.15){$z_1 z_2$};
\node at  (-.5,2.15){$z_0 z_1 z_2$};

\node at  (-.4,1.2){$A$};
\node at  (2.5,1.2){$A'$};

\node at  (.15,-1){$D'$};
\node at  (1.8,-1){$D$};

\node at (1,2.45){$B'$};
\node at (1,1.5){$C'$};

\node at (1,.4){$B$};

\node at (1,-.5){$C$};

\begin{scope}[xshift=-4cm, scale=2 ]

\draw[magenta, thick] (0,0)--(1,0);
\draw [magenta, thick](0,0)--(0,1);
\draw (1,0)--(0,1);
\draw[magenta, thick] (0,0)--(-1,0);
\draw (-1,0)--(0,1);
\draw (-1,0)--(0,-1);
\draw[magenta,thick] (0,-1)--(0,0);
\draw (0,-1)--(1,0);

\draw [green, thick](0,-1)--(-1,0);
\draw [orange, thick](1,0)--(0,1);

\draw [violet, thick](0,-1)--(1,0);
\draw [brown, thick](-1,0)--(0,1);

\node at (-.15,.5){\Small{$T_3^0$}};
\node at (-.35,.25){$D'$};

\node at (-.1,.75){$1$};
\node at (-.1,.15){$2$};
\node at (-.73,.15){$3$};

\node at (.35,.25){$C'$};
\node at (.15,.5){\Small{$T_0^0$}};

\node at (.1,.75){$1$};
\node at (.1,.15){$3$};
\node at (.73,.15){$2$};

\node at (-.15,.-.5){\Small{$T_1^0$}};
\node at (-.35,.-.25){$C$};

\node at (-.1,-.75){$2$};
\node at (-.1,-.12){$1$};
\node at (-.73,-.12){$3$};

\node at (.15,-.5){\Small{$T_2^0$}};
\node at (.35,-.25){$D$};

\node at (.1,-.75){$1$};
\node at (.1,-.12){$2$};
\node at (.73,-.12){$3$};

%
%
%
%

\end{scope}

\end{tikzpicture}
\caption{Left: Four cusp triangles from the unknotted cusp triangulation each corresponding to an ideal tetrahedron in the tetrahedral decomposition of the $(-2,3,8)$-pretzel link complement, Right: Ideal octahedron representing the $(-2,3,8)$-pretzel link complement with top faces $A$, $B$, $A'$, $B'$ and bottom faces $C$, $D$, $C'$, $D'$}
\label{octpret}
\end{figure}
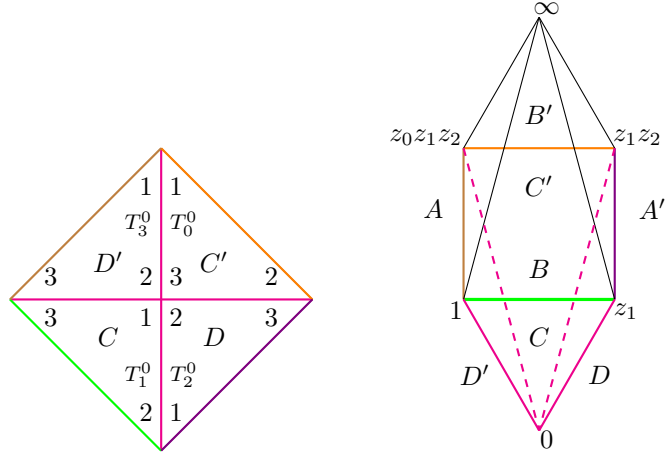

(We used Sage \cite{sagemath} for the computations below.) We will use $a$, $b$, $c$ and $d$ to denote respectively the isometries sending $A$ to $A'$,  $B$ to $B'$, $C$ to $C'$ and $D$ to $D'$. Poincar\'e polyhedron theorem implies that the fundamental group of $N$ has the following presentation:
\begin{equation}
\pi_1(N)=\left\langle c, d: a b=b a \right \rangle \text{ where } a=d^{-1} c d^{-1} \text{ and } b=cdc. \nonumber
\end{equation}
Hence, $\pi_1(N_{(p,q)})$ can then be seen as a group generated by $c$ and $d$ where 
\begin{align}
c=&\begin{pmatrix}-\frac{1}{1-z_1} & \frac{z_1}{1-z_1}\\ \frac{z_0+z_1-z_0 z_1}{(1-z_0) z_1} &-\frac{1}{1-z_0} \end{pmatrix},\nonumber\\
d=&\begin{pmatrix}-\frac{z_0}{\tau(1-z_0)} &-\frac{z_1}{\tau(1-z_1)}\\ \frac{z_0(1-z_1)}{\tau}+\frac{z_0}{\tau(z_0-1)}  & -\frac{z_1}{\tau(1-z_1)} \end{pmatrix}\nonumber.
\end{align}
We find that $\operatorname{Trace}(c)=-\frac{2-\sigma_1}{1-\sigma_1+\sigma_2}$, $\operatorname{Trace}(d)=-\frac{\sigma_1-2\sigma_2}{\tau(1-\sigma_1+\sigma_2)}$ and $\operatorname{Trace}(cd)=\frac{\sigma_2}{\tau}=\tau$. So, Lemma \ref{tau_sigma_pretzel} and \cite[Equation 3.25]{MacRe} implies that the trace field of  $\pi_1(N_{(p,q)})$ is $\Q(\tau)$, which, as mentioned in Fact \ref{pretzel_cusp_fld}, is the cusp field of $N_{(p,q)}$. Now, by applying \cite[Proposition 2.7]{NeumReid}) and arguing similarly as in the proof of Theorem \ref{Bcusptrace}, we see that the trace field, the invariant trace field and the cusp field of $N_{(p,q)}$ are all equal to $\Q(\tau)$. 
\end{proof}

\begin{remark}
As observed in \cite{HMPT}, there is an infinite family of hyperbolic twisted torus knots that can be realized as one-cusped Dehn fillings of the $(-2,3,8)$-pretzel link complement. So, the above theorems imply that these knot complements have the same trace field and cusp field. 
\end{remark}

\section{Arithmetic Data of the one-cusp Dehn fillings of the $(-2,3,8)$-pretzel link complement}\label{pretzel arithmetic}
SnapPy \cite{snappy} identifies the complement of the link \texttt{L13n5885} with the manifold \texttt{m125}. Let $\mathbf{m}_1^{125}$ and $\mathbf{l}_1^{125}$ respectively be the meridian and the longitude of cusp $1$ of the triangulation corresponding to \texttt{Manifold(`m125')} chosen by SnapPy \cite{snappy}. The isomorphism signature for this triangulation is $$\texttt{`eLPkbcdddlfffg\_BaabBaab'}. $$
Using the SnapPy command \texttt{Manifold(`m125').cusp\_neighborhood().view()}, one can see that the oriented generating pair $(\mathbf{m}_1^{125}, \mathbf{l}_1^{125})$ gives a square lattice for the corresponding peripheral group of cusp $1$ of the triangulation \texttt{Manifold(`m125')}. There are four possible oriented generating pairs giving square lattices. Denoting the set consisting of these four pairs by $S_{125}$, we see
$$S_{125}=\left \{(\mathbf{m}_1^{125}, \mathbf{l}_1^{125}), \left ((\mathbf{l}_1^{125})^{-1}, \mathbf{m}_1^{125} \right), \left ((\mathbf{m}_1^{125})^{-1}, (\mathbf{l}_1^{125})^{-1} \right), \left (\mathbf{l}_1^{125}, (\mathbf{m}_1^{125})^{-1} \right)  \right \}.$$

SnapPy \cite{snappy} tells us that the symmetry group of $\texttt{m125}$ only contains orientation preserving isometries and further, the stabilizer $\Delta$ of cusp $1$ in the symmetry group of $\texttt{m125}$ has order $4$ and is generated by a self-isometry $g_{\texttt{m125}}$ of $\mathbb{S}^3-\texttt{m125}$ such that  $g_{\texttt{m125}}(\mathbf{m}_1^{125})=(\mathbf{l}_1^{125})^{-1}$ and $g_{\texttt{m125}}(\mathbf{l}_1^{125})=\mathbf{m}_1^{125}$. This also means that  $g_{\texttt{m125}}$ sends $S_{125}$ to itself and $\Delta$  acts on $S_{125}$. This implies that $(p,q)$-Dehn filling on cusp $1$ of \texttt{Manifold(`m125')} with respect to one pair in $S_{125}$ is isometric to the $(p,q)$-Dehn filling of the same with respect to another pair in $S_{125}$. Consequently, the following fact holds. 
\begin{fact}\label{pretzel_nonnegativeDehn}
A $(p,q)$-Dehn filling on cusp $1$ of \texttt{Manifold(`m125')} with respect to a pair in $S_{125}$  is isometric to some $(p',q')$-Dehn filling of the same with respect to the same pair in $S_{125}$ for some non-negative integers $p'$ and $q'$. 
\end{fact}

By entering the command $$\texttt{Manifold(`m125').isomorphisms\_to(Manifold(`L13n5885'))}$$ on SnapPy \cite{snappy}, one can check that there is an isomorphism from the triangulation \texttt{Manifold(`m125')} to the triangulation \texttt{Manifold(`L13n5885') }sending cusp $1$ to cusp $0$ (and cusp $0$ to cusp $1$). Figure \ref{cusp0_pret} shows the cusp triangulation of cusp $0$ of  \texttt{Manifold(`L13n5885')} (i.e. the unknotted cusp of the $(-2,3,8)$ pretzel link complement). We will call this cusp $c_0'$. If we take $\mathbf{m}_0$ and $\mathbf{l}_0$ respectively as the red and blue directed edge as in Figure \ref{cusp0_pret}, we see that $(\mathbf{m}_0,\mathbf{l}_0)$ is an oriented generating pair of the cusp torus corresponding to $c_0'$ giving a square lattice. 
So, $(\mathbf{m}_0, \mathbf{l}_0)$ belongs to $S_{125}$. So, in view of Fact \ref{sym_fact} and \ref{pretzel_nonnegativeDehn}, we will assume that both $p$ and $q$ are non-negative in our analysis of the $(p,q)$-Dehn filling of the unknotted cusp with respect to $(\mathbf{m}_0, \mathbf{l}_0)$ in order to establish Theorem \ref{pretzel_arith_thm}.

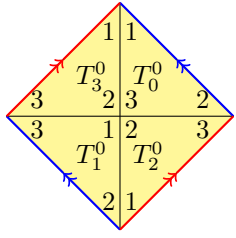
\begin{figure}
\begin{tikzpicture}[scale=1.5]
\fill[color=yellow!50] (-1,0)--(0,1)--(1,0)--(0,-1)--cycle;

\draw (0,0)--(1,0);
\draw (0,0)--(0,1);
\draw (1,0)--(0,1);
\draw (0,0)--(-1,0);
\draw (-1,0)--(0,1);
\draw (-1,0)--(0,-1);
\draw (0,-1)--(0,0);
\draw (0,-1)--(1,0);

\draw [blue, thick, ->>](0,-1)--(-.5,-.5);
\draw [blue, thick](-.5,-.5)--(-1,0);
\draw [blue, thick, ->>](1,0)--(.5,.5);
\draw [blue, thick](.5,.5)--(0,1);

\draw [red, thick, ->>](0,-1)--(.5,-.5);
\draw [red, thick](.5,-.5)--(1,0);
\draw [red, thick, ->>](-1,0)--(-.5,.5);
\draw [red, thick](-.5,.5)--(0,1);

\node at (-.25,.35){$T_3^0$};
\node at (-.1,.75){$1$};
\node at (-.1,.15){$2$};
\node at (-.73,.15){$3$};

\node at (.25,.35){$T_0^0$};
\node at (.1,.75){$1$};
\node at (.1,.15){$3$};
\node at (.73,.15){$2$};

\node at (-.25,.-.35){$T_1^0$};
\node at (-.1,-.75){$2$};
\node at (-.1,-.12){$1$};
\node at (-.73,-.12){$3$};

\node at (.25,-.35){$T_2^0$};
\node at (.1,-.75){$1$};
\node at (.1,-.12){$2$};
\node at (.73,-.12){$3$};

%
%
%
%

\end{tikzpicture}
\caption{Cusp triangulation of the unknotted cusp of the $(-2,3,8)$ pretzel link complement}
\label{cusp0_pret}
\end{figure}
Observe from Figure \ref{cusp0_pret} (and equations \ref{Pretzel_hol_eq} and \ref{z3eq_pret}) that the holonomy derivatives of the homology curves $\mathbf{m}_0$ and $\mathbf{l}_0$ are 
\begin{align*}
\mu (\mathbf{m}_0)&=-\zeta_1(z_3) \zeta_2(z_0) \zeta_1(z_1) \zeta_2(z_3)=\frac{\zeta_2(z_0) \zeta_1(z_1)}{z_3}=\frac{(1-z_0)^2 z_1}{(1-z_1)^2 z_0},\\
\mu (\mathbf{l}_0)&=-\zeta_1(z_2) \zeta_2(z_1) \zeta_1(z_1) \zeta_2(z_3)=\frac{\zeta_1(z_2) \zeta_2(z_3)}{z_1}=\frac{z_0}{z_1}.
\end{align*}

At the complete structure, we have the holonomy equation $\mu (\mathbf{l}_0)=\frac{z_0}{z_1}=1$. This implies $z_0=z_1$. This along with equations \ref{Pe2}, \ref{Pe0} and \ref{Pretzel_hol_eq} give us the following (which, as mentioned in the beginning of subsection \ref{pretzel_gluing}, can be verified from SnapPy \cite{snappy}). 
\begin{fact}
At the complete structure, i.e. when $N_{(p,q)}=N_{\infty}=N$, $z_0=z_1=z_2=z_3=\mathrm{i}$ and consequently, Fact \ref{pretzel_cusp_fld} implies that $\tau=\frac{2}{(1-\mathrm{i})^2}=\mathrm{i}$. 
\end{fact}

We will now investigate the $(p,q)$-Dehn fillings on the unknotted cusp with respect to the oriented generating pair $(\mathbf{m}_0, \mathbf{l}_0)$ where both $p$ and $q$ are non-negative. The corresponding $(p,q)$-Dehn filling equation is 

\begin{equation}\label{Dehneq_pret}
\mu(\mathbf{m}_0^p \mathbf{l}_0^q)=\frac{\left(1-z_0\right)^{2p}}{\left(1-z_1\right)^{2p}} \frac{z_1^{p-q}}{z_0^{p-q}}=1.
\end{equation}

As in the Berge manifold case, we will now follow the approach in \cite[Theorem 6.3]{NeumReid} and \cite[Proposition 6.6]{CDHMMW} to derive a similar result. We will show that cusp moduli  $\tau$ of $N_{(p,q)}$ is a unit algebraic integer when $p \ne 2q$. As in the Berge manifold case, in the proof below, we will employ the functions $s_d$ and $t_d$ of \cite{CDHMMW} from Subsection \ref{onecusp622}, which here will mean respectively $s_d(z_0,z_1)$ and $t_d(z_0,z_1)$, and use $g_{(p,q)}$, $h_{(p,q)}$ and $h'_{(p,q)}$ to denote the relevant polynomials to keep notational consistency with \cite{CDHMMW}. 
\begin{prop}\label{pret_algint}
If both $p$ and $q$ are non-negative integers such that $p\ne 2q$, then cusp moduli $\tau$ of the $(p,q)$-Dehn filling on cusp $c_0'$ of the $(-2,3,8)$-pretzel link complement with respect to $(\mathbf{m}_0, \mathbf{l}_0)$ is a unit non-rational algebraic integer. 
\end{prop}

\begin{proof}

We begin by defining 
\begin{align}
R_1&=\{(p,q): p \ge 0, q \ge 0, p-q\le 0\}\nonumber \\
R_2&=\{(p,q): p \ge 0, q \ge 0, p-q\ge 1, p\neq 2q\}.\nonumber
\end{align}
Note that the $R_1 \cup R_2=\left \{(p,q): p \text{ and } q \text{ are non-negative integers and } p \ne 2q\right \}$. 

Let $(p,q)\in R_1$. In this case, the Dehn filling equation \ref{Dehneq_pret} is equivalent to the equation: $(1-z_0)^{2p} z_0^{q-p}-(1-z_1)^{2p}z_1^{q-p}=0$. With the convention that $s_{-1}=t_{-1}=0$, we consider the following polynomials 
\begin{align*}
g_{(p,q)}&=\sum_{j=0}^{2p} \begin{pmatrix}2p\\ j \end{pmatrix} (-1)^j s_{q-p+j-1}, \text{ and}\\
h_{(p,q)}&=\sum_{j=0}^{2p} \begin{pmatrix}2p\\ j \end{pmatrix} (-1)^j t_{q-p+j-1}. 
\end{align*}

As in the proof of Theorem \ref{Balgint}, we simplify the Dehn filling equation (Equation \ref{Dehneq_pret}) to $g_{(p,q)}=0$ and we also have $\phi^{\ast} (h_{(p,q)})=g_{(p,q)}$. Now in this case, $\psi^{\ast} \left(h_{(p,q)}\right)=\tau^{-(q+p-1)}h'_{(p,q)}$ where $h'_{(p,q)}$ is a polynomial in $\Z[\tau]$ with leading term $\pm \tau^{3(q+p-1)}$ and constant term $\pm 1$. This implies that when $(p,q) \in R_1$, $\tau$ is a non-rational unit algebraic integer. 

Now, let $(p,q) \in R_2$. In this case, we will take 
\begin{align*}
g_{(p,q)}&=\sum_{j=0}^{p-q-1} \begin{pmatrix}2p\\ j \end{pmatrix} (-z_0 z_1)^j s_{p-q-j-1}-\sum_{j=p-q+1}^{2p} \begin{pmatrix}2p\\ j \end{pmatrix} (-1)^j (z_0 z_1)^{p-q}s_{j-p+q-1}, \\
h_{(p,q)}&=\sum_{j=0}^{p-q-1} \begin{pmatrix}2p\\ j \end{pmatrix} (-\sigma_2)^j t_{p-q-j-1}-\sum_{j=p-q+1}^{2p} \begin{pmatrix}2p\\ j \end{pmatrix} (-1)^j (\sigma_2)^{p-q}t_{j-p+q-1}. 
\end{align*}

When $(p,q)\in R_2$, Equation \ref{Dehneq_pret} gives us $(1-z_0)^{2p} z_1^{p-q}-(1-z_1)^{2p} z_0^{p-q}=0$, which is equivalent to $(z_0-z_1) g_{(p,q)}=0$. Since $z_0=z_1$ corresponds to the complete structure at cusp $1$, the $(p,q)$-Dehn filling equation in this case is equivalent to $g_{(p,q)}=0$. As before, $\phi^{\ast} \left(h_{(p,q)}\right)=g_{(p,q)}$. The highest  $\tau$ degree term in $\psi^{\ast}(\sigma_2^j t_{p-q-j-1})$ is $\pm \tau^{2p-2q-2}$ and lowest $\tau$ degree term in $\psi^{\ast}(\sigma_2^j t_{p-q-j-1})$ is $\pm\tau^{3j-p+q+1}$. On the other hand, the highest and lowest $\tau$ degree terms in $\psi^{\ast}(\sigma_2^{p-q} t_{j-p+q-1})$ are $\pm \tau^{2j-2}$ and $\pm \tau^{3p-3q-j+1}$ respectively. So, the highest $\tau$ degree term in $\psi^{\ast}(h_{(p,q)})$ is $\pm \tau^{4p-2}$ (coming from the second sum term corresponding to $j=2p$) and lowest degree term in $\psi^{\ast}(h_{(p,q)})$ is either $\pm \tau^{-p+q+1}$ or $\pm \tau^{p-3q+1}$ (coming from respectively the first sum term corresponding to $j=0$ or the second sum term corresponding to $j=2p$). Note that since $p\neq 2q$, $\tau^{-p+q+1}\neq \tau^{p-3q+1}$. So, when $(p,q)\in R_2$, $\tau$ is a unit non-rational algebraic integer as well. This proves the proposition. 
\end{proof}
We will now investigate the cusp fields of the $(p,2p)$-Dehn fillings of the $(-2,3,8)$-pretzel link complement. We will use the approach taken in \cite[Lemma 6.13]{CDHMMW}. This proposition will play a very crucial role in the proof of Theorem \ref{pretzel_arith_thm}. 
\begin{prop}\label{p2pfilling}
The only natural number $p$ for which the cusp field $\mathbb{Q}(\tau)$ of the $(p,2p)$-Dehn filling on cusp $c_0'$ of the $(-2,3,8)$-pretzel link complement with respect to $(\mathbf{m}_0, \mathbf{l}_0)$ is quadratic imaginary is $p=1$. \end{prop}
\begin{proof}
Let $p\in \mathbb{N}$. The Dehn filling equation \ref{Dehneq_pret} for the $(p,2p)$-filling reduces to 
\begin{equation*}
\left(\mu(\mathbf{m}_0 \mathbf{l}_0^2)\right)^p=\left(\frac{(1-z_0)^2z_0}{(1-z_1)^2z_1} \right)^p=1.
\end{equation*}
Moreover, $\mu(\mathbf{m}_0 \mathbf{l}_0^2)$ is a primitive $p$-th root of unity $\zeta_p^k$ for some integer $k$ such that $1\le k\le p$ and $gcd(k,p)=1$, where $\zeta_p=\cos(2\pi/p)+ \mathrm{i} \sin(2\pi/p)$. Now, 
\begin{equation*}
\mu(\mathbf{m}_0 \mathbf{l}_0^2)=\frac{(1-z_0)^2z_0}{(1-z_1)^2z_1}=\frac{(1-z_0)^2z_0^2}{(1-z_1)^2\tau^2}. 
\end{equation*}
This gives us that for some integer $k$ such that $1\le k\le p$ and $gcd(k,p)=1$, we can write, 
\begin{align}
2 \cos(k\pi/p)&=\frac{(1-z_0)z_0}{(1-z_1)\tau}+\frac{(1-z_1)\tau}{(1-z_0)z_0}\nonumber\\
&=\frac{\sigma_1-2(\sigma_1^2-2\sigma_2)+\sigma_1^3-3\sigma_1 \sigma_2}{(1-\sigma_1+\sigma_2)\tau}\nonumber\\
&=\frac{(\tau^9+3\tau^8+\tau^7-3\tau^6-4\tau^5-2\tau^4+\tau^3+3\tau^2+\tau-1)\tau}{-\tau^3(\tau^2-1)\tau}\nonumber\\
&=\frac{\tau^7+3\tau^6+2\tau^5-2\tau^3-2\tau^2-\tau+1}{-\tau^3}\label{pretzelp2p}.
\end{align}
Note that it follows from the discussion in the beginning of Section \ref{pretzel arithmetic} that $(1,2)$-Dehn filling on cusp $c_0'$ with respect to $(\mathbf{m}_0, \mathbf{l}_0)$ is isometric to the SnapPy triangulation $$\texttt{Manifold(`m125(0,0)(1,2)').filled\_triangulation()}$$ which one can check from SnapPy \cite{snappy} to be isometric to \texttt{m003}. Manifold \texttt{m003} is arithmetic since it can be identified on SnapPy \cite{snappy} as \texttt{otet02\_{00000}} of the tetrahedral census of \cite{FGGTV}. Furthermore, when $p=1$, the above equation can be simplified to the following equation in $\tau$.  
\begin{align*}
\tau^7+3\tau^6+2\tau^5-4\tau^3-2\tau^2-\tau+1=(\tau-1)(\tau^2+2\tau-1)(\tau^2+\tau+1)^2=0. 
\end{align*}
Since this Dehn filling is hyperbolic, $\tau$ satisfies $\tau^2+\tau+1=0$, i.e., $\tau=\frac{-1+\sqrt{3}\mathrm{i}}{2}$. 

Now assume that $p>1$. Suppose $\Q(\tau)$ is quadratic imaginary. Then $[\Q(\tau): \Q(\cos(k\pi/p))]$ is 2 since $\Q(\cos(k\pi/p))$ is a real field. So, $\Q(\cos(k\pi/p))=\Q$. It is a standard fact that $[\Q(\cos(\pi/p)):\Q]=\frac{\varphi(2p)}{2}$ for $p>1$ where $\varphi$ is Euler's totient function (we note that it was observed and used in \cite[Proof of Lemma 6.13]{CDHMMW}). Moreover, $[\Q(\cos(\pi/p)):\Q]=[\Q(\cos(k\pi/p)):\Q]$. So, $\varphi(2p)=2$ and consequently, $p\in \{2,3\}$. When $p=2$, Equation \ref{pretzelp2p} becomes $\tau^7+3\tau^6+2\tau^5-2\tau^3-2\tau^2-\tau+1=0$ and for $p=3$, Equation \ref{pretzelp2p} reduces to either $\tau^7+3\tau^6+2\tau^5-\tau^3-2\tau^2-\tau+1=0$, or $\tau^7+3\tau^6+2\tau^5-3\tau^3-2\tau^2-\tau+1=0$. One can check in Sage \cite{sagemath} that all three of these polynomials are irreducible, which contradicts that  $\Q(\tau)$ is quadratic imaginary. So, $\Q(\tau)$ is not quadratic imaginary for $p>1$. This concludes the proposition. 
\end{proof}

We note that Proposition \ref{pret_algint} plays a similar role to that of Proposition \ref{Balgint} of the Berge manifold case. However, it leaves out the fillings of the form $N_{(2q,q)}$. But, as the proof of Theorem \ref{pretzel_arith_thm} below shows, this $p=2q$ case can be dealt using the approach taken in the proof of \cite[Lemma 6.13]{CDHMMW}. For the $p\ne 2q$ case in the proof below, we follow the approach in the proof of \cite[Corollary 6.14]{CDHMMW}. 

\begin{thm}\label{pretzel_arith_thm}
\pretzelarith
\end{thm}

\begin{proof}
Fact \ref{sym_fact} and \ref{pretzel_nonnegativeDehn} imply that it is enough to test the arithmeticity of the hyperbolic orbifolds $N_{(p,q)}$ obtained by $(p,q)$-Dehn filling cusp $c_0'$ with respect to $(\mathbf{m}_0, \mathbf{l}_0)$, where both $p$ and $q$ are non-negative integers. 

First consider the case when $(p,q)\in R_3=\{(p,q): p \ge 0, q \ge 0, p=2q\}$. Equation \ref{Dehneq_pret} in this case becomes
\begin{equation*}
\left (\frac{(1-z_0)^4 z_1}{(1-z_1)^4 z_0} \right )^q=1.
\end{equation*}
In other words, $\left(\mu(\mathbf{m}_0^2 \mathbf{l}_0)\right)^q=1$, which implies $\mu(\mathbf{m}_0^2 \mathbf{l}_0)=\zeta_q^k$ where 
$$\zeta_q=\cos(2\pi/q)+ \mathrm{i} \sin(2\pi/q)$$ and $k$ is an integer such that $1\le k\le q$ and $gcd(k,q)=1$. On the other hand, we note that 
\begin{equation*}
\mu(\mathbf{m}_0^2 \mathbf{l}_0)=\frac{(1-z_0)^4 z_1}{(1-z_1)^4 z_0}=\frac{(1-z_0)^4 z_1^2}{(1-z_1)^4 \sigma_2}=\frac{(1-z_0)^4 z_1^2}{(1-z_1)^4 \tau^2}.
\end{equation*}
So, the above equation gives us 
\begin{equation}
\frac{(1-z_0)^2 z_1}{(1-z_1)^2 \tau}+\frac{(1-z_1)^2 \tau}{(1-z_0)^2 z_1}=2 \cos(k\pi/q), \label{cosformula_pret}
\end{equation}
for integer $k$ such that $1\le k\le q$ and $gcd(k,q)=1$. A quick calculation renders that 
\begin{align*}
\frac{(1-z_0)^2 z_1}{(1-z_1)^2 \tau}+\frac{(1-z_1)^2 \tau}{(1-z_0)^2 z_1}&=\frac{\sigma_1-8\sigma_2+6\sigma_1 \sigma_2-4\sigma_1^2\sigma_2+8\sigma_2^2+\sigma_1^3 \sigma_2-3\sigma_1 \sigma_2^2}{(1-\sigma_1+\sigma_2)^2 \tau}\\
&=\frac{\tau^9+3\tau^8-\tau^7-7\tau^6-\tau^5+3\tau^4+\tau^3+3\tau^2-2}{\tau^4-2\tau^2+1}\\
&=\tau^5+3\tau^4+\tau^3-\tau^2-2.
\end{align*}

When $q=1$, the above equation simplifies to $\tau^5+3\tau^4+\tau^3-\tau^2=0$. All roots of this equation are real. Hence, the $(2,1)$-filling is not hyperbolic. 

We will now show that for $q>1$, $\Q(\tau)$ is not quadratic imaginary. The argument that follows is similar to the $p>1$ case in Proposition \ref{p2pfilling}. More specifically, assuming $\Q(\tau)$ is quadratic imaginary, the observation in \cite[Proof of Lemma 6.13]{CDHMMW} of the standard fact that $[\Q(\cos(\pi/q)):\Q]=\frac{\varphi(2q)}{2}$ will lead us to $q\in \{2,3\}$. When $q=2$, Equation \ref{cosformula_pret} and calculation below it gives the equation $\tau^5+3\tau^4+\tau^3-\tau^2-2=0$. However, one can use Sage \cite{sagemath} to see that $\tau^5+3\tau^4+\tau^3-\tau^2-2$ is irreducible over $\Q$. So, $[\Q(\tau):\Q]=5$, which contradicts that $\Q(\tau)$ is quadratic imaginary. If $q=3$, then we have either $\tau^5+3\tau^4+\tau^3-\tau^2-3=0$, or $\tau^5+3\tau^4+\tau^3-\tau^2-1=0$. But, it can be checked on Sage \cite{sagemath} that both $\tau^5+3\tau^4+\tau^3-\tau^2-3$ and $\tau^5+3\tau^4+\tau^3-\tau^2-1$ are irreducible over $\Q$ as well and so $[\Q(\tau):\Q]=5$ contradicting that $\Q(\tau)$ is quadratic imaginary. So, when $(p,q)\in R_3$, $\Q(\tau)$ is not quadratic imaginary. 

Now, consider the case when $p \ne 2q$, where $p$ and $q$ are both non-negative integers. 
Assume $\Q(\tau)$ is quadratic imaginary. Proposition \ref{pret_algint} tells us that $\tau$ is a non-rational unit algebraic integer. Since $\tau$ is a cusp moduli of $N_{(p,q)}$ by \cite[Proposition 1.6]{CDM}, it lies in the upper-half plane. So, by Fact \ref{quad_unit}, we have $\tau\in \left \{\mathrm{i}, \frac{1+\sqrt{3}\mathrm{i}}{2}, \frac{-1+\sqrt{3}\mathrm{i}}{2}\right\}$. We first rule out $\tau=\mathrm{i}$ as it corresponds to the complete structure at cusp $1$. 

Now, $\tau=\frac{1+\sqrt{3}\mathrm{i}}{2}$ implies $z_0+z_1=\tau^2+\tau+1-\tau^{-1}=2\tau^2+\tau+1=3 \tau-1$. Substituting $z_1$ by $\frac{\tau^2}{z_0}$, we get the formula: $z_0^2+z_0-1=(3z_0-1)\tau$. Solving for $z_0$ and $z_1$ from the previous formula, we find that
$\{z_0,z_1\}$, is $$\{0.148402943598350 + 2.36455298678789\mathrm{i}, 0.351597056401650 + 0.233523224565427\mathrm{i} \}.$$ Taking the norms of $\frac{(1-z_0)}{(1-z_1)}$ and $\frac{z_0}{z_1}$ one can check that these choices of $z_0$ and $z_1$ don't satisfy Equation \ref{Dehneq_pret} for non-negative $p$ and $q$ unless $p=q=0$. 

On the other hand, if $\tau=\frac{-1+\sqrt{3}\mathrm{i}}{2}$, then $z_0+z_1=\sigma_1=\tau^2+\tau+1-\tau^{-1}=-\tau^2$. Since $z_0 z_1=\tau^2$, we get the formula: $z_1^2+\tau^2z_1+\tau^2=0$. Solving for $z_1$ and $z_0$ from this equation gives that $\{z_0,z_1\}$ is 
$$\{1.12196442695186 + 1.05375577424138\mathrm{i}, -0.621964426951856 - 0.187730370456945\mathrm{i}\}.$$
These choices of $z_0$ and $z_1$ satisfy Equation \ref{Dehneq_pret} for non-negative integers $p$ and $q$ if and only if $q=2p$. However, Proposition \ref{p2pfilling} forces $p=1$ for $\mathbb{Q}(\tau)$ to be quadratic imaginary. 

We note that the cusp field $\Q(\tau)$ is not a real field. Now $\Q(\tau)$ is not quadratic imaginary for any non-negative integers $p$ and $q$ other than $p=1$ and $q=2$. Moreover, as  can be observed from SnapPy \cite{snappy}, the $(1,2)$-filling is the arithmetic manifold \texttt{m003} (see the paragraph below Equation \ref{pretzelp2p}). We can now use  \cite[Proposition 4.4 (a)]{NeumReid} and Theorem \ref{tracecusp_pret} (or \cite[Proposition 2.7]{NeumReid}) to conclude the theorem. 
\end{proof}
\begin{remark}
Note that \texttt{m003} is the sister of the figure-eight knot complement and has volume $2v_0$ where $v_0$ is the volume of a regular ideal tetrahedron. 
\end{remark}
\begin{remark}
Denote the meridian and the longitude of cusp $0$ of \texttt{Manifold(`L13n5885’)} chosen by SnapPy \cite{snappy} by $\mathbf{m}_0^{PL}$ and $\mathbf{l}_0^{PL}$ respectively. These are the preferred meridian and longitude corresponding to the unknotted component of the link on the right of Figure \ref{link_pics}. By entering the command $$\texttt{Manifold(`m125’).isomorphisms\_to(Manifold(`L13n5885’))}$$ on SnapPy \cite{snappy}, we see that there is an isomorphism sending $\mathbf{m}_1^{125}$ to $-3\mathbf{m}_0^{PL}-\mathbf{l}_0^{PL}$ and $\mathbf{l}_1^{125}$ to  $4\mathbf{m}_0^{PL}+\mathbf{l}_0^{PL}$. In particular, the arithmetic manifold \texttt{m003} obtained by $(1,2)$-filling on cusp $c_0'$ is obtained by $(5,1)$-filling along the preferred meridian-longitude pair $(\mathbf{m}_0^{PL}, \mathbf{l}_0^{PL})$ on the unknotted component of the link on right of Figure \ref{link_pics}. These can also be verified by entering $$\texttt{Manifold(`L13n5885(5,1)(0,0)').filled\_triangulation().identify()}$$ on SnapPy \cite{snappy}. 
\end{remark}
\begin{remark}
We remark that $\tau=\frac{1+\sqrt{3}\mathrm{i}}{2}$ provides solutions of $p$ and $q$ in Equation \ref{Dehneq_pret} outside non-negative integers\textemdash it gives $p=-2q$ where $q\le 0$. However, note that for $q\le 0$, the $(-2q,q)$-Dehn fillings are isometric to the $\left(-q,-2q\right)$-Dehn fillings by Fact \ref{pretzel_nonnegativeDehn}, which are taken care of by Proposition \ref{p2pfilling}.
\end{remark}

\section{Hidden symmetries and Effectivization} \label{hidd_symm}
There have been a wide variety of studies on which one-cusped Dehn fillings of the Berge manifold can admit hidden symmetries. The earliest such trace back to Hoffman \cite[Proof of Theorem 1.1]{Hoffman_3comm}, where he showed that for integers $m$ and $n$ such that $n>0$ and $gcd(n,7)=1$, all but finitely many $(n,m)$-Dehn filling with respect to a standard framing\footnote{See \cite{Rolfsen} for details on standard framing.} on the unknotted cusp of the Berge manifold lack hidden symmetries. He strengthened this result further in \cite[Theorem 6.1]{Hoffman_hidden} by showing that when both $gcd(n,7)$ and $gcd(n,m)$ are $1$, none of such hyperbolic $(n,m)$-Dehn fillings on the unknotted cusp of the Berge manifold admit hidden symmetries. The proof technique in \cite[Theorem 1.1]{Hoffman_3comm} uses a limiting argument concerning orbifold fundamental groups whereas the proof of \cite[Theorem 6.1]{Hoffman_hidden} uses degree bounds for covers of rigid one-cusped orbifolds. More recently, using separate techniques, \cite[Example 5.3]{CDM} and \cite[Theorem 8.4]{Mon} imply that at-most finitely many Dehn fillings on the unknotted cusp of the Berge manifold can have hidden symmetries. The work in \cite[Example 5.3]{CDM} centers around the non-geometric isolation of cusps and and that in \cite[Theorem 8.4]{Mon} upon the study of certain order $3$ symmetries in relevant horoball packings of $\mathbb{H}^3$. 

On the other hand, using the findings from \cite{AabD}, it was shown in \cite[Example 5.2]{CDM} (by virtue of non-geometric isolation of cusps) that all but finitely many knot complements obtained by one-cusped Dehn filling of the $(-2,3,8)$-pretzel link complement lack hidden symmetries.

Here, we can give an effective statement on the relation of such Dehn fillings with hidden symmetries. Observe that the proof of Theorem \ref{berge_arith_thm} shows that the cusp field of a one-cusped Dehn filling of the Berge manifold can't be quadratic imaginary. On the other hand, Theorem \ref{pretzel_arith_thm} shows that the cusp field of a one-cusped Dehn filling of the $(-2,3,8)$-pretzel link complement is quadratic imaginary if and only if it is the tetrahedral manifold \texttt{m003}. However, the only knot complement in the commensurability class of \texttt{m003} is the figure eight knot complement which does not cover \texttt{m003} (since both have the same volume). Recall that the cusp field is a commensurability invariant (see \cite[Subsection 2.3]{NeumReid}) and consequently we can now apply \cite[Proposition 9.1]{NeumReid} to conclude the following.
\begin{thm}\label{hidd_sym_effect}
\hidsym
\end{thm}

\bibliographystyle{plain}
\bibliography{sistersDehn_bib}

\end{document}